\documentclass[12 pt]{amsart}
\usepackage{amssymb,amsfonts,amsthm,color}
\usepackage[makeroom]{cancel}
\usepackage{graphicx}
\usepackage{bm}
\newtheorem{thm}{Theorem}[section]
\newtheorem{cor}[thm]{Corollary}
\newtheorem{lem}[thm]{Lemma}
\newtheorem{pro}[thm]{Proposition}

\theoremstyle{definition}

\theoremstyle{remark}

\numberwithin{equation}{section}
\begin{document}
\title{Firing rate and spatial correlation in a stochastic neural
  field model}
\author{Yao Li}
\address{Yao Li: Department of Mathematics and Statistics, University
  of Massachusetts Amherst, Amherst, MA, 01002, USA}
\email{yaoli@math.umass.edu}

\author{Hui Xu}
\address{Hui Xu: Department of Mathematics, Amherst College, Amherst,
  MA, 01002, USA}
\email{huxu18@amherst.edu}
\thanks{Yao Li and Hui Xu were partially supported by the University of
  Massachusetts Amherst FRG/HEG grant.}


\begin{abstract}
This paper studies a stochastic neural field model that is extended
from our previous paper \cite{li2017well}. The neural field model consists
of many heterogeneous local populations of neurons. Rigorous results on
the stochastic stability are proved, which further imply the well-definedness of quantities
including mean firing rate and spike count correlation. Then we devote
to address two main topics: the comparison with
mean-field approximations and the spatial correlation of spike
count. We showed that partial synchronization of spiking activities is
a main cause for discrepancies of mean-field
approximations. Furthermore, the spike count correlation between local
populations are studied. We find that the spike count correlation
decays quickly with the distance between corresponding local
populations. Some mathematical justifications of the mechanism of this
phenomenon is also provided. 
\end{abstract}

\maketitle

\section{Introduction}

In mathematical neuroscience, numerous neuronal models have been proposed and
studied. Many models such as the Hodgkin-Huxley model aim to accurately
describe the realistic biophysics process. When more anatomical and
physiological details, such as ionic channels, dendritic tree, spatial
structure of axon, and synapse, are involved, the model usually becomes too complex to study, especially when applying to a large-scale network. On the other end of the spectrum,
there are many mean field models such as Wilson-Cowan model
\cite{wilson1972excitatory, wilson1973mathematical} and
various population density models \cite{haskell2001population,
  cai2004effective, cai2006kinetic} that model the
behavior of a neural population by a few coarse-grained
variables. It is easier to use a mean-field model to describe the
neuronal activities in a large area of the cortex. However, mean field
models usually assume that neuronal interactions are weak, which
cause some inevitable discrepancies. 

Consider a large area of the cortex, such as several hundreds of
hypercolumns in the primary visual cortex. In order to study
experimentally observed phenomena such as surround suppression, 
neuronal models at this scale are necessary. Needless to say, it is not realistic to
model each of millions of neurons by detailed models like the
Hodgkin-Huxley model. Some coarse graining is necessary for these
large scale problems. On the other hand, it is usually not clear how 
much information a mean field model could preserve. The mechanism of
such discrepancy is also little studied. In our previous paper
\cite{li2017well}, these questions are partially answered for a stochastic
neuron model that models a homogeneous and densely connected neuronal
population. Several mean field approximations of the stochastic model
studied in \cite{li2017well} are exactly solvable. We showed that the
strong excitation-inhibition interplay during the partially
synchronous spike volley, called a {\it multiple firing event},
contributes significantly to the discrepancy of the mean-field
approximation. It is believed that such multiple firing event is
related to the Gamma rhythm in the cortex.

The aim of this paper is two-fold. The first half of this paper serves as
an extension to our previous paper \cite{li2017well}. We generalize the
stochastic model studied in \cite{li2017well} to a neural field model that
describes the neuronal activities of neurons in a large and
heterogeneous domain. More precisely, we consider the coupling of
finitely many local neuron populations, each of which is described by
a stochastic model studied in \cite{li2017well}. This extension is
necessary because the real cerebral cortex consists of numerous
relatively homogeneous local structures, while the neuronal activities in
different local structures can be very different. Take the visual
cortex as an example again. Responses to stimulus orientation are very
different in orientation columns with different orientation
preferences, even if they are spatially close to each other
\cite{kaschube2010universality, hubel1995eye}. Similar as in
\cite{li2017well}, we proved the stochastic stability of the Markov process
generated by the model, which implies that many quantities like the
mean firing rate are well-defined. Then we proposed two exactly
solvable mean-field approximations of our neural field model. The
discrepancy of the mean-field approximations is analyzed. Same as in
the homogeneous population model, the discrepancy of a mean field
model is mainly caused by the emergent coordinated neuronal
activities, and is significantly exacerbated when the neuron spikes
become synchronized.

A more important result presented in this paper is the spatial
correlation of spike counts in the neural field. This is not studied
in our previous paper \cite{li2017well}. We found that the multiple firing
event in nearest-neighbor local populations are highly
correlated. However, this correlation decays quickly with increasing
distance between two local populations. This is consistent with the
experimental observations that the Gamma rhythm is very local
\cite{goddard2012gamma, lee2003synchronous, menon1996spatio}. Two
analytical studies are carried out to investigate this spatial
correlation. We first propose an ODE model that describe the activity
during a multiple firing event. This ODE model shows that the strong
excitatory and inhibitory current during a multiple firing event at
one local population is very likely to induce a similar multiple
firing event in its neighbor local populations. This explains the
mechanism of the spatial correlation. Then we studied the possible
mechanism of the spatial correlation decay. One salient feature of the
multiple firing event is its diversity. When starting from the same
profile, the spike count of a multiple firing event can have very high
volatility, measured by the coefficient of variation. The volatility
significantly decreases when the multiple firing event is close to a
synchronous spiking event. We believe this diversity at least
partially contributes to the quick decay of the spatial correlation. This
is verified by our numerical simulation result, in which the most
synchronous network has the slowest decay of the spatial correlation.

The organization of this paper is as follows. Section 1 is the
introduction. We provide the mathematical description of the neural
field model in Section 2. Section 3 is devoted to the proof of the
stochastic stability and several useful corollaries. Then we provide
five network examples with distinct features in Section 4 for later
investigations. Two mean-field approximations are presented in Section
5. We also analyze their discrepancies in the same section. Section 6
is devoted to the spatial correlation. We demonstrate both phenomena
and mechanism of the spatial correlation between spike counts of
different local populations. Section 7 is the conclusion.

%skip it for now

\section{Stochastic model of many interacting neural populations}
We consider a stochastic model that describes the interaction of many
local populations in the cerebral cortex. Each local population
consists of hundreds of interacting excitatory and inhibitory
neuron. The setting of this model is quit generic, but it aims to
describe some of the spiking activities of realistic brain parts. In
particular, one can treat a local population as an orientation column
in the primary visual cortex. We will only prescribe the rules of
external inputs and interactions between neurons. All spatial and temporal correlated
spiking activities in this model are emergent from interactions of
neurons.

\subsection{Model description.} We consider an $M \times N$ array of
local populations of neurons, each of which are homogeneous and densely connected by
local circuitries. In addition, neurons in nearest neighbor of local
populations are connected. Each local population consists of
$N_{E}$ excitatory neurons and $N_{I}$ inhibitory neurons. Similar to
the treatment in \cite{li2017well}, we have the following assumptions in
order to describe the activity of this population by a Markov process.

\begin{itemize}
  \item The membrane potential of each neuron can only take finitely
    many discrete values.
\item The external current to each neuron is in the form of
  independent Poisson processes. The rate of Poisson kicks to all
  neurons of the same type in each local population is a constant.
\item A neuron spikes when its membrane potential reaches a given
  threshold. After a spike, a neuron stays at a refractory state for
  an exponentially distributed amount of time.
\item When an excitatory (resp. inhibitory) spike occurs at a local population, a set of
  postsynaptic neurons from this local population and its nearest
  neighbor local populations are randomly chosen. After an
  exponenitally distributed random time, the membrane
  potential of each chosen postsynaptic neuron goes up (resp. down). 
\end{itemize}

More precisely, we consider $M \times N $ local populations
$\{ L_{m,n}\}$ with $m = 1 , \cdots, M$ and $n = 1, \cdots,
N$. $L_{m,n }$ and $L_{m', n'}$ are considered to be nearest neighbors
if and only if $m = n'$, $| n - n'| = 1$ or $|m - m'| = 1$, $n =
n'$. For each $(m,n)$, denote the set of indices of its nearest
neighbors local populations by $\mathcal{N}(m,n)$. Each population $L_{m,n}$ consists of $N_{E}$ excitatory neurons, labeled
$$
  (m, n, 1), (m, n, 2), \cdots, (m, n, N_{E})
$$
and $N_{I}$ inhibitory neurons, labeled
$$
  (m, n, N_{E}+1), (m, n, N_{E} + 2), \cdots, (m, n, N_{E} + N_{I}) \,.
$$
In other words, each neuron has a unique label $(m, n, k)$. The
membrane potential of neuron $(m, n, k)$, denoted $V_{(m, n, k)}$,
takes value in a finite set
$$
  \Gamma := \{ - M_{r}, - M_{r} + 1, \cdots, -1, 0, 1, 2, \cdots, M \}
  \cup \{\mathcal{R} \} \,,
$$
where $M$ and $M_{r}$ are two integers for the threshold potential and
the inhibitory reversal potential, respectively. When $V_{(m, n, k)}$
reaches $M$, the neuron fires a {\it spike} and reset its membrane
potential to $\mathcal{R}$. After an exponentially distributed amount
of time with mean $\tau_{R}$, $V_{(m, n, k)}$ leaves $\mathcal{R}$ and jumps to $0$.

We first describe the external current that models the input from
other parts of the brain or the sensory input. As in \cite{li2017well}, the
external current received by a neuron is modeled by a homogeneous
Poisson process. The rate of this Poisson process is identical for
the same type of neurons in the same local population. Neurons in
different local population receive different external current, which makes this
model spatially heterogeneous. More precisely, we
assume the rate of such Poisson kick of an excitatory (resp. inhibitory) neuron in local population
$L_{m,n}$ to be $\lambda_{m,n}^{E}$ (resp. $\lambda_{m,n}^{I}$). When a kick is
received by neuron $(m,n, k)$ and it is not at state $\mathcal{R}$,
$V_{(m,n,k)}$ jumps up by $1$ immediately. If it reaches $M$, a spike
is fired. Neurons at state $\mathcal{R}$ do not respond to external
kicks. 

The rule of interactions among neurons is the following. We assume that a
postsynaptic kick from an E (resp. I) neuron takes effect after an exponentially
distributed amount of time with mean $\tau_{E}$ (resp. $\tau_{I}$). To model this
delay effect, we describe the state of neuron $(m, n,
k)$ by a triplet $(V_{(m,n,k)}, H^{E}_{(m,n,k)}, H^{I}_{(m,n,k)})$,
where $H^{E}_{(m, n, k)}$ (resp. $H^{I}_{(m,n,k)}$) denote the number
of received E (resp. I) postsynaptic kicks that has not yet taken
effect. Further, we assume that the delay of postsynaptic kicks are
independent. Therefore, two exponential clocks corresponding to
excitatory and inhibitory kicks are associated to each neuron, with
rates $H^{E}_{(m, n, k)} \tau_{E}^{-1}$ and $H^{I}_{(m,
  n, k)} \tau_{I}^{-1}$ respectively. When the clock corresponding to
excitatory (resp. inhibitory) kicks rings, an excitatory
(resp. inhibitory) kick takes effect according to the rules described in the following paragraph. 

Let $Q, Q'\in \{E, I\}$. When a postsynaptic kick from a neuron of
type $Q$ takes effect at a neuron of
type $Q'$ after the delay time as described above, the membrane potential
of the postsynaptic neuron, say neuron $(m,n,k)$, jumps
instantaneously by a constant $S_{Q', Q}$ if $V_{(m,n,k)} \neq \mathcal{R}$. No
change happens if $V_{(m,n,k)} = \mathcal{R}$. If after the jump we
have $V_{(m,n,k)} \geq M$, neuron $(m,n,k)$ fires a spike and jumps to
state $\mathcal{R}$. Same as in \cite{li2017well}, if constant $S_{Q', Q}$
is not an integer, we let $u$ be a Bernoulli random variable with
$\mathbb{P}[u = 1] = S_{Q', Q} - \left \lfloor S_{Q', Q} \right \rfloor$
that is independent of all other random variables in the model. Then
the magnitude of the postsynaptic jump is set to be the random number
$\left \lfloor S_{Q', Q} \right \rfloor + u$. 

It remains to describe the connectivity within and between local
populations. We assume that each local population is densely
connected and homogeneous, while different local populations are
heterogeneous in a way that the external currents are different. For
example, each local population can be thought as an orientation column
in the primary visual cortex. Hence nearest-neighbor populations
receive very different external drives due to their different
orientational preferences. Same as in \cite{li2017well}, the connectivity in
our model is random and time-dependent. For $Q, Q' \in \{E, I\}$, we
choose two parameters $P_{Q, Q'}, \rho_{Q, Q'} \in [0, 1]$ representing
the local and external connectivity respectively. When a neuron of
type $Q'$ in a local population $L_{m,n}$ fires a spike, every neuron
of type $Q$ in the local population $L_{m,n}$ is postsynaptic with
probability $P_{Q, Q'}$, while every neuron of type $Q$ in the
nearest-neighbor populations $L_{m', n'}$ receives this postsynaptic
kick with probability $\rho_{Q, Q'}$. In other words, neurons of the
same type in the same local population are assumed to be
indistinguishable. 

\subsection{Common parameters for simulations.} Although our
theoretical results are valid for all parameters, in numerical
simulations we will stick to the following set of parameters in order
to be consistent with \cite{li2017well}. Throughout this paper, we assume that $N_{E} = 300$, $N_{I} =
100$ for the size of local populations, $M = 100$, $M_{r} = 66$ for the
thresholds, $P_{EE} = 0.15$, $P_{IE} = P_{EI} = 0.5$ and $P_{II} =
0.4$ for local conductivities. The conductivities to nearest neighbors
are assumed to be proportional to the corresponding local
conductivities. We set two parameters $\mbox{ratio}_{E}$ and
$\mbox{ratio}_{I}$ and let $\rho_{QE} = \mbox{ratio}_{E} P_{QE}$,
$\rho_{QI} = \mbox{ratio}_{I} P_{QI}$ for $Q = I, E$. Further, we
assume that $\mbox{ratio}_{I} = 0.6
\mbox{ratio}_{E}$ as inhibitory neurons are known to be more
``local''. The strengths of
postsynaptic kicks are assumed to be $S_{EE} = 5$, $S_{IE} = 2$,
$S_{EI} = 3$, and $S_{II} = 3.5$. The length
of refractory period is set as $\tau_{\mathcal{R}} = 4
\mbox{ms}$. Since AMPA synapses act faster than GABA synapses, in
general $\tau_{E}$ is assumed to be faster than $\tau_{I}$. Values of
$\tau_{E}$ and $\tau_{I}$ are two changing parameters that are used to
control the degree of synchrony of the network. External drive rates $\lambda_{m,n}^{E}$ and
$\lambda_{m,n}^{I}$ are determined when describing examples with
different spatial structures.

\section{Stochastic stability and proofs}
The aim of this section is to show the stochastic stability of the
model presented in Section 2. As a corollary, we have  well-defined
and computable local and global firing rates, and the spike count correlation between local
populations. 

\subsection{Statement of results.} The neural field model described
above generates a Markov jump process $\Phi_{t}$ on a countable state space
$$
  \mathbf{X} = ( \Gamma \times \mathbb{Z}_{+} \times
  \mathbb{Z}_{+})^{M\times N \times (N_{E} + N_{I})} \,.
$$
The state of neuron $(m, n, k)$ is given by the triplet $(V_{(m,n,k)},
H^{E}_{(m,n,k)}, H^{I}_{(m,n,k)})$, where $V_{(m,n,k)} \in \Gamma$ and
$H^{E}_{i}, H^{I}_{i} \in \mathbb{Z}_{+} := \{ 0, 1, 2, \cdots\}$. The
transition probabilities of $\Phi_{t}$ are denoted by $P^{t}(
\mathbf{x}, \mathbf{y})$, i.e., 
$$
  P^{t}( \mathbf{x}, \mathbf{y}) = \mathbb{P}[ \Phi_{t} = \mathbf{y}
  \,|\, \Phi_{0} = \mathbf{x}] \,.
$$
If $\mu$ is a probability distribution on $\mathbf{X}$, the left
operator of $P^{t}$ acting on $\mu$ is 
$$
  \mu P^{t}( \mathbf{x}) = \sum_{\mathbf{y} \in \mathbf{X}} \mu(
  \mathbf{y}) P^{t}( \mathbf{y}, \mathbf{x}) \,.
$$
Similarly, the right operator of $P^{t}$ acting on a real-valued
function $\eta: \mathbf{X} \rightarrow \mathbb{R}$ is 
$$
  P^{t} \eta( \mathbf{x}) = \sum_{ \mathbf{y} \in \mathbf{X}} P^{t}(
  \mathbf{x}, \mathbf{y}) \eta( \mathbf{y}) \,.
$$
Finally, for any probability measure $\mu$ and real-valued function
$\eta$ on $\mathbf{X}$, we take the convention that
$$
  \mu(\eta) = \sum_{\mathbf{x} \in \mathbf{X}}
  \eta(\mathbf{x})\mu(\mathbf{x}) \,.
$$

For the stochastic stability, we mean the existence, uniqueness, and
ergodicity of the invariant measure for $\Phi_{t}$. Note that
$\mathbf{X}$ has countably infinite states. Hence Markov chains on
$\mathbf{X}$ need not admit an invariant probability measure. 

Define the total number of pending excitatory (resp. inhibitory) kicks
at a state $\mathbf{x} \in \mathbf{X}$ as 
$$
  H^{E}( \mathbf{x}) = \sum_{m = 1}^{M} \sum_{n = 1}^{N} \sum_{k =
    1}^{N_{E} + N_{I}} H^{E}_{(m, n, k)}
$$
and
$$
  H^{I}( \mathbf{x}) = \sum_{m = 1}^{M} \sum_{n = 1}^{N} \sum_{k =
    1}^{N_{E} + N_{I}} H^{I}_{(m, n, k)} \,.
$$
Further we let $U( \mathbf{x}) = H^{E}( \mathbf{x}) +
H^{I}(\mathbf{x}) + 1$. For any signed measure on the Borel
$\sigma$-algebra of $\mathbf{X}$, denoted by $\mathcal{B}(X)$, we
define the $U$-weighted total variation norm to be
$$
  \| \mu \|_{U} = \sum_{ \mathbf{x} \in \mathbf{X}} U( \mathbf{x})
  |\mu( \mathbf{x}) | \,,
$$
and let 
$$
  L_{U}( \mathbf{X}) = \{ \mu \mbox{ on } \mathbf{X} \, | \, \| \mu
  \|_{U} < \infty \} \,.
$$
In addition, for any measurable function $\eta( \mathbf{x})$ on
$\mathbf{X}$, we let
$$
  \sup_{ \mathbf{x} \in \mathbf{X}} \frac{|\eta( \mathbf{x})|}{U(\mathbf{x})}
$$
be the $U$-weighted supreme norm. 

\begin{thm}
\label{invariant}
$\Phi_{t}$ admits a unique invariant probability measure $\pi \in
L_{U}( \mathbf{X})$. In addition, there exist constants $C_{1}$,
$C_{2} > 0$ and $r \in (0, 1)$ such that
\begin{itemize}
  \item (a) for any initial distribution $\mu \in L_{U}( \mathbf{X})$,
$$
  \| \mu P^{t} - \pi \|_{U} \leq C_{1} r^{t} \| \mu - \pi \|_{U} \,;
$$
\item (b) for any measurable function $\eta$ with $\| \eta \|_{U} <
  \infty$, 
$$
  \| P^{t} \eta - \pi (\eta) \|_{U} \leq C_{2} r^{t} \| \eta - \pi
  (\eta) \|_{U} \,.
$$
\end{itemize}
\end{thm}

Theorem \ref{invariant} guarantees that the local/global firing rate
and the spike count correlation between local populations are well
defined. Let $L_{m,n}$ be a given local population. For $Q \in \{ E, I
\}$, let $N^{Q}_{(m,n)}([a, b])$ be the number of 
neuron spikes fired by type $Q$ neurons in $L_{m,n}$ on the time interval $[a, b]$. As discussed in
\cite{li2017well}, the mean firing rate of the local population $(m,n)$ is
defined to be 
$$
  F^{Q}_{m,n} = \frac{1}{T}\mathbb{E}_{\pi}[ N^{Q}_{(m,n)}([0, T])] \,,
$$
where $\mathbb{E}_{\pi}$ is the expectation with respect to the
invariant probability measure $\pi$. This definition is independent of
$T$ by the invariance of $\pi$. 

Let $T$ be a fixed time window, we can further define the covariance of spike count
between $Q_{1}$-population in $L_{m,n}$ and $Q_{2}$-population in
$L_{m', n'}$ as
$$
  \mbox{cov}_{T}^{Q_{1}, Q_{2}}(m,n; m', n') =
  \mathbb{E}_{\pi}[N^{Q_{1}}_{(m,n)}([0, T]) N^{Q_{2}}_{(m',n')}([0, T]) ] -
  T^{2}N_{Q_{1}}N_{Q_{2}} F^{Q_{1}}_{m,n} F^{Q_{2}}_{m',n'}\,.
$$

The Pearson correlation coefficient of spike count can be defined
similarly. For $Q_{1}$-population in $L_{m,n}$ and $Q_{2}$-population in
$L_{m', n'}$, we have
$$
  \rho^{Q_{1}, Q_{2}}_{T}(m,n,m',n') = \frac{\mbox{cov}_{T}^{Q_{1},
      Q_{2}}(m,n,m',n')}{\sigma^{Q_{1}}_{T}(m,n)
   \cancel{ \sigma^{Q_{2}}_{T}(m,n)} \sigma^{Q_{2}}_{T}(m',n')}\,,
$$
where
$$
  \sigma^{Q}_{T}(m,n) = \sqrt{\mbox{var}_{\pi}(N^{Q}_{(m,n)}([0, T]))}
$$
for $Q \in \{E, I\}$. 

One can also consider the correlation of the total spike count between
two local populations. Let $N_{(m,n)}([0, T])$ be the number of
excitatory and inhibitory spikes produced by $L_{m,n}$ on $[0, T]$
when starting from the steady state. Then the correlation
$\mbox{cov}_{T}(m,n,m',n')$ and the Pearson correlation coefficient
$\rho_{T}(m,n,m',n') $ can be defined analogously.

The following corollaries implies that the mean firing rate and the spike
count correlation are computable.

\begin{cor}
\label{cor1}
For $Q \in \{E, I\}$, $1 \leq m \leq M$, and $1 \leq n \leq N$, the
local firing rate $F^{Q}_{m,n} < \infty$. In addition, for any initial
value $\mathbf{x} \in X$,
$$
\lim_{T \rightarrow \infty} \frac{N^{Q}_{(m,n)}([0, T])}{N_{Q}T}  = F^{Q}_{(m,n)}
$$
almost surely. 
\end{cor}

\begin{cor}
\label{cor2}
For $Q_{1}, Q_{2} \in \{E, I\}$, $1 \leq m, m' \leq M$, and $1 \leq n,
n' \leq N$, the covariance $\mbox{cov}_{T}^{Q_{1}, Q_{2}}(m,n; m', n')
< \infty$. In addition, for any initial value $\mathbf{x} \in X$, 
$$
 \lim_{K \rightarrow \infty} \frac{1}{K} \sum_{ k = 0}^{K-1}
  N^{Q_{1}}_{(m,n)}([kT, (k+1)T)) N^{Q_{2}}_{(m',n')}([kT, (k+1)T)) =
  \mathbb{E}_{\pi}[N^{Q_{1}}_{(m,n)}([0, T]) N^{Q_{2}}_{(m',n')}([0,
  T])] 
$$
almost surely. In other words $\mbox{cov}_{T}^{Q_{1}, Q_{2}}(m,n; m',
n')$ is computable. 
\end{cor}

\begin{cor}
\label{cor3}
For $1 \leq m, m' \leq M$, and $1 \leq n,
n' \leq N$, the covariance $\mbox{cov}_{T}(m,n; m', n')
< \infty$. In addition, for any initial value $\mathbf{x} \in X$, 
$$
  \lim_{K \rightarrow \infty} \frac{1}{K}\sum_{ k = 0}^{K-1}
  N_{(m,n)}([kT, (k+1)T]) N_{(m',n')}([kT, (k+1)T)) =
  \mathbb{E}_{\pi}[N_{(m,n)}([0, T]) N_{(m',n')}([0,
  T])] 
$$
almost surely. In other words $\mbox{cov}_{T}(m,n; m', n')$ is
computable. 
\end{cor}

\subsection{Probabilistic Preliminaries} Let $\Psi_{n}$ be a Markov
chain on a countable state space $(X, \mathcal{B})$  
with transition kernels $\mathcal{P}(x, \cdot)$. Let $W: X
\rightarrow [1, \infty)$ be a real-valued function. The following general results
on geometric ergodicity is well known.

Assume $\Psi_{n}$ satisfies the following conditions.

\begin{itemize}
  \item[(a)] There exist constants $K \geq 0$ and $\gamma \in (0, 1)$
    such that
$$
  (\mathcal{P}W)(x) \leq \gamma W(x) + K
$$ 
for all $x \in X$.
\item[(b)] There exists a constant $\alpha \in (0, 1)$ and a
  probability distribution $\nu$ on $\mathbf{X}$ so that
$$
  \inf_{x\in C} \mathcal{P}(x, \cdot) \geq \alpha \nu(\cdot) \,,
$$
with $C = \{x \in X \, | \, W(x) \leq R \}$ for some $R > 2K/ (1 -
\gamma)$, where $K $ and $\gamma$ are from (a).
\end{itemize}

\medskip

The following result was proved in \cite{meyn2009markov} and
\cite{hairer2011yet} using different methods. Note that the original
result in \cite{meyn2009markov, hairer2011yet} is for a generic
measurable state space. The result applies to countable state space
with the Borel $\sigma$-algebra (which is essentially the discrete
$\sigma$-algebra). 

\begin{thm}
\label{hairer}
Assume (a) and (b). Then $\Psi_{n}$ admits a unique invariant measure
$\pi \in L_{W}(X)$. In addition, there exist constants $C, C' > 0$ and
$r \in (0, 1)$ such that (ii) for all $\mu, \nu \in L_{W}(X)$, 
$$
  \| \mu \mathcal{P}^{n} - \nu \mathcal{P}^{n} \|_{W} \leq C r^{n} \|
  \mu - \nu \|_{W} \,,
$$ 
and (i) for all $\xi$ with $\|\xi\|_{W} < \infty$,
$$
  \| \mathcal{P}^{n} \xi - \pi( \xi) \|_{W} \leq C' r^{n} \| \xi -
  \pi(\xi) \|_{W} \,.
$$
\end{thm}

We also need the following law of large numbers for martingale
difference sequence to prove the corollary. 

\begin{thm}[Theorem 3.3.1 of \cite{stout1974almost}]
\label{mds}
Let $X_{n}$ be a martingale difference sequence with respect to
$\mathcal{F}_{n}$. If 
$$
  \sum_{n = 1}^{\infty} \frac{\mathbb{E}[|X_{n}|^{2}]}{n^{2}} < \infty \,,
$$
then
$$
  \frac{1}{N}\sum_{n = 1}^{N} X_{n} \rightarrow 0 \quad a.s. 
$$
\end{thm}

\subsection{Proof of main results.}
For a step size $h > 0$ that will be described later, we define the time-$h$ sample chain as
$\Phi^{h}_{n} = \Phi_{nh}$. The superscript $h$ is dropped when it leads
to no confusion. Recall that $U(\mathbf{x}) = H^E(\mathbf{x}) +
H^I(\mathbf{x})+1$. The following two lemmas verify conditions (a) and
(b) for Theorem \ref{hairer}.

\begin{lem}
\label{cond1}
For $h > 0$ sufficiently small, there exist constants $K
> 0$ and $\gamma \in (0, 1)$, such that 
$$
  P^{h}U \leq \gamma U + K \,.
$$
\end{lem}

This proof is similar to that of Lemma 2.4 of \cite{li2017well}. We include
it for the sake of completeness of this paper. 

\begin{proof}
During $(0, h]$, let $N_{out}$ be the number of pending kicks from
    $H^E(\mathbf{x})$ and $H^I(\mathbf{x})$ that takes effect and
  $N_{in}$ be the number of new spikes produced. We have 
$$
  P^{h}U( \mathbf{x}) = \mathbb{E}_{\mathbf{x}}[ U( \Phi_{h})] = U(
  \mathbf{x}) - \mathbb{E}_{\mathbf{x}}[ N_{out}] + \mathbb{E}_{\mathbf{x}}[N_{in}] \,.
$$

The probability that an excitatory (resp. inhibitory) pending kick
takes effect on $(0, h]$ is $(1-e^{- h/ \tau^E})$ (resp. $(1-e^{- h/
  \tau^I})$). Hence for $h$ sufficiently small, we have
$$
 \mathbb{E}_{\mathbf{x}}[ N_{out}] \ge (H^E(\mathbf{x}) + H^I(\mathbf{x})) ( 1 -
 e^{-h/\max\{\tau^E, \tau^I\}} ) \geq  \frac{1}{2 \max\{\tau^E,
   \tau^I\}}  \ h \ (U(\mathbf{x})-1) \,.
$$

For a neuron $(m,n,k)$, after each spike it spends an exponential time with mean
$\tau_{\mathcal R}$ at state $\mathcal R$. Hence the number of spikes
produced by neuron $(m,n,k)$ is at most

$$
  1 + \mathbb{E}[
  \mbox{Pois}(h/\tau_{\mathcal R})] = 1 + h/\tau_{\mathcal R} \,, 
$$
where $\mbox{Pois}(\lambda)$ is a Poisson random variable with rate
$\lambda$. Hence
$$
  \mathbb{E}_{\mathbf{x}}[N_{in}] \leq MN(N_{E} + N_{I})\cdot(1 + h/\tau_{\mathcal R}) \,.
$$
The proof is completed by letting 
$$
\gamma = 1 - h/(2 \max\{\tau^E, \tau^I\}) \qquad \mbox{and} \qquad
K =  MN(N_{E} + N_{I})\cdot (1 + h/\tau_{\mathcal R})\ +\frac{h}{2 \max\{\tau^E,\tau^I\}}.
$$
\end{proof}

For $b \in \mathbb{Z}_{+}$, let
$$
C_{b} = \{  \mathbf{x} \in \mathbf{X} | H^E( \mathbf{x}) + H^I( \mathbf{x})
\leq b \} \,. 
$$

\begin{lem}
\label{cond2}
Let $\mathbf{x}_{0}$ be the state that $H^E = H^I = 0$ and  $V_{(m,n,k)} = \mathcal R$
for all $1 \leq m \leq M$, $1 \leq n \leq N$, and $1 \leq k \leq N_{E}
+ N_{I}$. Then for any $h > 0$, there exists a constant $\delta =
\delta (b, h)$ depending on $b$ such that there exists a constant $c$
depending on $b$ and $h$ such that, 
$$
  P^{h}(\mathbf{x}, \mathbf{x}_{0}) > c \quad \mbox{ for all } \mathbf{x} \in C_{b}
$$
\end{lem}
\begin{proof}
For each $\mathbf{x} \in C_{b}$, it is sufficient to construct an
event that moves from $\mathbf{x}$ to $\mathbf{x}_{0}$ with a uniform
positive probability. Below is one of many possible constructions.

\begin{itemize}
  \item[(i)] On $(0, h/2]$, a sequence of external Poisson kicks drives each $V_{(m,n,k)}$
   to the threshold value $M$, hence puts $V_{(m,n,k)}=\mathcal R$. Once at
$\mathcal R$, $V_{(m,n,k)}$ remains there before $t = h$. In addition,
no pending kicks takes effect on $(0,h/2]$.
\item[(ii)] All pending kicks at state $\mathbf{x}$ take effect on
  $(h/2, h]$. Obviously this has no effect on membrane potentials.
\end{itemize}

Since $b$ is bounded, the number of pending kicks is less than $b +
MN(N_{E} + N_{I}) < \infty$. It is easy to see that this event happens
with positive probability.
\end{proof}

Lemmas \ref{cond1} and \ref{cond2} together imply Theorem \ref{invariant}.

\begin{proof}[Proof of Theorem \ref{invariant}] 
Choose step size $h$ as in Lemma \ref{cond1}. By Lemmata
\ref{cond1}, \ref{cond2} and Theorem
\ref{hairer}, $\Phi^{h}$ admits a unique invariant probability measure $\pi_{h}$ in $L_{U}(
\mathbf{X})$.

\medskip

We then show that $\pi_{h}$ is invariant under $\Phi_t$ for any
$t > 0$. This can be done by proving the following ``continuity at zero''
condition, which means for
any probability measure $\mu$ on $\mathbf{X}$,
$$
  \lim_{t \rightarrow 0}\| \mu P^{t} - \mu \|_{TV} = 0 \,.
$$

\medskip

For any small $\epsilon > 0$, there exists $b < \infty$ and (small)
$\delta > 0$ such that if
$U = \{ \mathbf{x} \in \mathbf{X} \,|\, H^{E}( \mathbf{x}) +
H^{I}(\mathbf{x}) < b \}$, then (i) $\mu(U) > 1 - \epsilon/4$ and (ii) $\mathbb{P}[ \mbox{ no clock rings on } [0, \delta) ] \geq 1
- \epsilon/4$. For any set $A \subset \mathbf{X}$, we have
\begin{eqnarray*}
(\mu P^{\delta}) (A)  & =  & \sum_{\mathbf{x} \in \mathbf{X}}
P^{\delta}( \mathbf{x} , A) \mu( \mathbf{x} )\\
&=& \sum_{\mathbf{x}\in U \cap A} P^{\delta}(\mathbf{x}, A) \mu( \mathbf{x} ) + \sum_{\mathbf{x} \in U -
  A} P^{\delta}(\mathbf{x}, A) \mu( \mathbf{x} ) + \sum_{\mathbf{x} \in U_{c}} P^{\delta}(\mathbf{x}, A) \mu( \mathbf{x}) \\
&=& \mu(U \cap A) - a_{1} + a_{2} + a_{3} \,,
\end{eqnarray*}
where
\begin{eqnarray*}
a_{1} & =  & \sum_{\mathbf{x} \in U \cap A} (1 -
             P^{\delta}(\mathbf{x}, A) )\mu(\mathbf{x}) \leq
             \frac{\epsilon}{4} \mu( U \cap A) \leq
             \frac{\epsilon}{4}\\ 
a_{2} &=& \sum_{\mathbf{x} \in U \setminus A} P^{\delta}(\mathbf{x},
          A) \mu( \mathbf{x}) \leq \frac{\epsilon}{4} \mu( U \setminus
          A) \leq \frac{\epsilon}{4}\\ 
a_{3} &=&\sum_{\mathbf{x} \in U^{c}}P^{\delta}(\mathbf{x},
          A) \mu( \mathbf{x}) \leq \mu(U^c) \leq \frac{\epsilon}{4} \,.
\end{eqnarray*}
Further, $\mu(A) - \mu(U \cap A) \leq \mu(U^{c}) <
\frac{\epsilon}{4}$. Hence
$$
  \epsilon > \sup_{A \subset \mathbf{X}} | (\mu P^{\delta})(A) - \mu(A) | \geq \| \mu P^{\delta} - \mu \|_{TV} \,.
$$
This implies the ``continuity at zero'' condition.

Notice that $\pi_{h}$ is invariant for any $\Psi^{hj/k}_{n}$, where
$j, k \in \mathbb{Z}^{+}$ (Theorem 10.4.5 of \cite{meyn2009markov}). Assume $t/h \notin
\mathbb{Q}$ without loss of generality. By the density of orbits in irrational rotations, there 
exists sequences $a_{n}$, $b_{n} \in \mathbb{Z}^{+}$ such that
$$
  d_{n} := t - \frac{a_{n}}{b_{n}}h \searrow 0 \,.
$$
Therefore,
$$
\| \pi_{h}P^{t} - \pi_{h} \|_{TV} = \lim_{n \rightarrow \infty} \|
\pi_{h} P^{\frac{a_{n}}{b_{n}} h} 
  P^{d_{n}}  - \pi_{h} \| =   \lim_{n \rightarrow \infty}
\| \pi_{h}P^{d_{n}} - \pi_{h} \|_{TV}  = 0
$$
by the ``continuity at zero'' condition. Hence
$\pi_{h}$ is invariant with respect to $P^{t}$. 

\medskip

It remains to prove the exponential convergence for any $t > 0$. By
Lemma \ref{cond1}, there exists $B = K + 1 < \infty$ such that $P^{t}U \leq
B U$ for all $t < h$. Let $n = \left \lfloor t/h \right \rfloor$ and $d = t
- hn$. We have

\begin{eqnarray*}
  \| \mu P^{t} - \nu P^{t}\|_{U} & = & \| (\mu P^{d}) P^{nh} - (\nu
  P^{d}) P^{nh}\|_{U}\\
&  = & C r^{n} \cdot \| \mu P^{d} - \nu P^{d} \|_{U}\  \leq \ BC r^{n}
  \| \mu - \nu \|_{U} 
\end{eqnarray*}
and
\begin{eqnarray*}
  \| P^{t} \xi - \pi(\xi)\|_{U} & = & \| P^{nh} (P^{d} \xi)- P^{nh} (P^{d}
  \pi(\xi))\|_{U}\\
 & = & C r^{n} \cdot \| P^{d} \xi - P^{d} \pi(\xi) \|_{U} \ \leq \ BC r^{n}
  \| \xi - \pi(\xi) \|_{U} \,.
\end{eqnarray*}
This completes the proof.
\end{proof}

\medskip
\begin{proof}[Proof of Corollary \ref{cor1}] By the invariance of
  $\pi$, for any local population $L_{m,n}$ and any $Q \in \{E, I\}$, we have
$$
  F^{Q}_{m,n} = \mathbb{E}_{\pi}[N^{Q}_{(m,n)}([0, 1])] \,.
$$
For every $\mathbf{x} \in \mathbf{X}$ we have
$$
  \mathbb{E}_{\mathbf{x}}[N^{Q}_{m,n}([0, 1])] \leq N_{Q}(1 + \mathbb{E}[ \mbox{
    Pois}(1/\tau_{\mathcal R})] )= N_{Q}(1 + 1/\tau_{\mathcal R}) \,. 
$$ 
Thus $F^{Q}_{m,n} = \mathbb{E}_{\pi}[N^{Q}_{(m,n)}([0, 1])] < \infty$.

It remains to prove the law of large number. Without loss of
generality let $Q = E$. Let
$$
  \xi( \mathbf{x}) = \sum_{k = 1}^{N_{E}} \mathbf{1}_{\{ V_{m,n,k}  =
    \mathcal{R}\}} \,.
$$
Then by the Ergodic Theorem, for every $\mathbf{x}$ and almost every
sample path $\Phi_{t}$ with initial condition $\mathbf{x}$, we have
$$
  \lim_{T \rightarrow \infty} \frac{1}{T} \int_{0}^{T} \xi( \Phi_{t})
  \mathrm{d}t = \pi( \xi) = N_{E} \tau_{\mathcal{R}} F^{E}_{m,n}\,.
$$
In addition, the time duration that neurons stay at $\mathcal{R}$ are
independent. Then by law of large numbers,
$$
  \lim_{T\rightarrow \infty} \frac{\int_{0}^{T} \xi( \Phi_{t})
    \mathrm{d}t}{N^{E}_{(m,n)}([0, T])} = \tau_{\mathcal{R}} \,.
$$
This completes the proof. 

\end{proof}

\medskip
\begin{proof}[Proof of Corollary \ref{cor2}]
Consider the time-$T$ sample chain of $\Phi_{k}:=\Phi_{kT}$. Define an auxiliary process $Y_{k}, k \geq 0$ such that
$$
  Y_{k} = N^{Q_{1}}_{(m,n)}( [kT, (k+1)T )) N^{Q_{2}}_{(m', n')}([kT,
  (k+1)T)) \,.
$$
Let $Z_{k} = \mathbb{E}_{\Phi_{k}}[Y_{k}]$. It is easy to see that
$Z_{k}$ is a measurable function of $\Phi_{k}$. In addition, $Y'_{k} =
Y_{k} - Z_{k}$ is a martingale difference sequence because
$\mathbb{E}[Y'_{k} \,|\, \mathcal{F}_{k}] = 0$, where $\mathcal{F}_{k}$
is the $\sigma$-field generated by $\{\Phi_{0}, \cdots,
\Phi_{n}\}$. By Holder's inequality, we have
$$
  \mathbb{E}[|Y'_{k}|^{2}] \leq \mathbb{E}[|Y_{k}|^{2}] \leq
  \mathbb{E}[ (N^{Q_{1}}_{(m,n)}( [kT, (k+1)T
  )))^{4}]^{1/2}\mathbb{E}[ (N^{Q_{2}}_{(m',n')}( [kT, (k+1)T
  )))^{4}]^{1/2} \,.
$$
Same as before, we have
$$
  \mathbb{E}[ (N^{Q_{1}}_{(m,n)}( [kT, (k+1)T
  )))^{4}] \leq \mathbb{E}[ (1 +
  \mbox{Pois}(\frac{T}{\tau_{\mathcal{R}}}))^{4}] < \infty
$$
and
$$
  \mathbb{E}[ (N^{Q_{2}}_{(m',n;)}( [kT, (k+1)T
  )))^{4}] \leq \mathbb{E}[ (1 +
  \mbox{Pois}(\frac{T}{\tau_{\mathcal{R}}}))^{4}] < \infty \,.
$$
Then by the law of large numbers of martingale difference sequence, we
have
$$
  \frac{1}{K}\sum_{n = 0}^{K-1} (Y_{k} - Z_{k}) \rightarrow 0\,.
$$
In addition, since $Z_{k}$ is an observable of $\Phi_{k}$, by Ergodic
Theorem we have
$$
  \frac{1}{K}\sum_{k = 0}^{K-1}Z_{k} \rightarrow
  \mathbb{E}_{\pi}[Z_{0}] = \mathbb{E}_{\pi}[N^{Q_{1}}_{m,n}([0, T))
  N^{Q_{2}}_{m',n'}([0, T)) ] = \mathbb{E}_{\pi}[N^{Q_{1}}_{m,n}([0, T])
  N^{Q_{2}}_{m',n'}([0, T]) ] 
$$
almost surely. Hence 
$$
  \frac{1}{K}\sum_{n = 0}^{K-1} Y_{k} \rightarrow \mathbb{E}_{\pi}[N^{Q_{1}}_{m,n}([0, T])
  N^{Q_{2}}_{m',n'}([0, T]) ] \,.
$$
\end{proof}

\medskip 

The proof of Corollary \ref{cor3} is identical to that of Corollary \ref{cor2}.

\section{Numerical examples with different local and global spiking
  patterns} 
The neural field model proposed in this paper is spatially
heterogeneous. Hence its spiking pattern consists of two factors:
the local synchrony and the spatial correlation. By adjusting
parameters, we can change not only the degree of partial synchrony within a
local population, but also the spike count correlation between local
populations. These local and global spiking pattern are emergent from
network activities. One goal of this paper is to interpret how
these emergent patterns arise from the interaction of
neurons. Obviously it is not practical to test all possible
parameters and discover all emergent spiking patterns. Instead, we will
demonstrate the following five numerical examples representing five
different local and global degrees of synchrony. 

\subsection{Parameters of examples.} 

We will use common parameters prescribed in Section 2.2. In addition,
we use $N = M = 3$ in all five examples. For the sake of simplicity, $9$ populations are label as
$1, 2, \cdots, 9$ from upper left corner to lower right corner. In
order to have heterogeneous external drive rates, we assume that $\lambda^{E}_{m,n} = \lambda^{I}_{m,n} =
\lambda_{Even}$ if $(n-1)*M + m$ is even, and $\lambda^{E}_{m,n} = \lambda^{I}_{m,n} =
\zeta\lambda_{Even}$ if $(n-1)M + m$ is odd. This alternating
external drive rates is similar to the realistic model of visual
cortex if a local population models an orientation column. The main varying parameters
in our numerical examples are the strength of nearest-neighbor
connectivity $\mathit{ratio}_{E}$, the ratio of external drive rates in
nearest neighbors $\zeta$, and the synapse delay time after the
occurrence of a spike $\tau_{E}, \tau_{I}$. We first follow the idea
of \cite{li2017well} to produce three examples with ``homogeneous'',
``regular'', and ``synchronized'' patterns respectively by varying the
synapse delay time. Then we change
$\mbox{ratio}_{E}$ and external drive rates for the ``regular''
network to produce two more examples with different global
synchrony. As in \cite{li2017well}, we replace a single $\tau^{E}$ by two
synapse times $\tau^{EE}$ and $\tau^{IE}$ to
denote the expected delay times after an excitatory spike takes effect
in excitatory and inhibitory neurons, respectively. 

\begin{itemize}
  \item The {\bf ``homogeneous''} neural field, denoted by {\bf HOM} in
    the figures:
$$
  \tau_{EE} = 4 \mbox{ ms},\quad \tau_{IE} = 1.2 \mbox{ ms}, \quad
  \tau_{I} = 4.5 \mbox{ ms}, \quad \mathit{ratio}_{E} = 0.1, \quad \zeta
  = 11/12 \,.
$$
\item The {\bf ``synchronized''} neural field, denoted by {\bf SYN} in
  the figures: 
$$
    \tau_{EE} = 0.9 \mbox{ ms},\quad \tau_{IE} = 0.9 \mbox{ ms}, \quad
  \tau_{I} = 4.5 \mbox{ ms}, \quad \mathit{ratio}_{E} = 0.15, \quad \zeta
  = 11/12 \,.
$$
\item The {\bf ``regular''} neural field with weak nearest neighbor
  connectivity, denoted by {\bf REG1} in the figures:
$$
    \tau_{EE} = 1.6 \mbox{ ms},\quad \tau_{IE} = 1.2 \mbox{ ms}, \quad
  \tau_{I} = 4.5 \mbox{ ms}, \quad \mathit{ratio}_{E} = 0.05, \quad \zeta
  = 11/12 \,.
$$
\item The {\bf ``regular''} neural field with strong nearest neighbor
  connectivity, denoted by {\bf REG2} in the figures:
$$
    \tau_{EE} = 1.6 \mbox{ ms},\quad \tau_{IE} = 1.2 \mbox{ ms}, \quad
  \tau_{I} = 4.5 \mbox{ ms}, \quad \mathit{ratio}_{E} = 0.15, \quad \zeta
  = 11/12 \,.
$$
\item The {\bf ``regular''} neural field with strong nearest neighbor
  connectivity and fluctuating external drive rates, denoted by {\bf REG3} in the figures:
$$
    \tau_{EE} = 1.6 \mbox{ ms},\quad \tau_{IE} = 1.2 \mbox{ ms}, \quad
  \tau_{I} = 4.5 \mbox{ ms}, \quad \mathit{ratio}_{E} = 0.15, \quad \zeta
  = 1/2 \,.
$$
\end{itemize}

\subsection{Numerical results for five example neural fields.} 
We present numerical simulation result of the five example
networks. The rastor plots generated by networks {\bf HOM} and {\bf
  SYN} are not very different from what we have presented in
\cite{li2017well}. The {\bf HOM} network produces homogeneous spike trains
in all local populations and the {\bf SYN} network produces largely synchronized neuron activities in all local populations. Since neuron activities in different local populations have the same pattern, we only present the
rastor plot of the central local population $L_{2,2}$ for these two
examples in Figure \ref{synhom}.

\begin{figure}[h]
\centerline{\includegraphics[width = \linewidth]{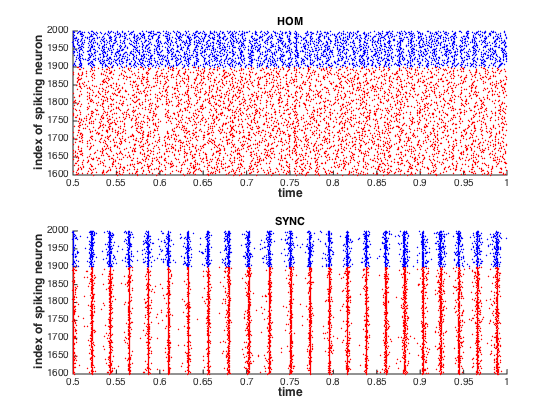}}
\caption{Rastor plot of the central local population of {\bf SYN} and
  {\bf HOM} networks. The drive rate at even-indexed local populations
  is $\lambda_{Even} = 6000$ spikes/sec. The index of neuron $(m,n,k)$
  is $[(n-1)M + m](N_{E} + N_{I}) + k$.}
\label{synhom}
\end{figure}

The three {\bf REG} networks are much more interesting as we can see
different spike count correlations among different local populations
when parameters change. With higher $\mathit{ratio}_{E}$
($\mathit{ratio}_{E} = 0.15$), the spike
activities in all $9$ blocks are largely correlated (Figure \ref{reg}
middle panel). When $\mathit{ratio}_{E} = 0.05$, much less correlation
is seen. And the rastor plot also looks less synchronized (Figure
\ref{reg} left panel). If the long range connectivity remains
$\mathit{ratio}_{E} = 0.15$ but we drive odd-indexed local populations
only half strong as even-indexed local populations, the spike count
correlation is between the previous two cases (Figure \ref{reg} right
panel). From the three rastor plots presented in Figure \ref{reg}, we
can conclude that both stronger long range connectivity and more
homogeneous drive rate contribute to a more correlated spiking
pattern among different local populations. A natural question is
how such correlated spiking activity changes when two local populations
that are further apart. We will extensively investigate this problem
in Section 6. 

\begin{figure}[h]
\centerline{\includegraphics[width = \linewidth]{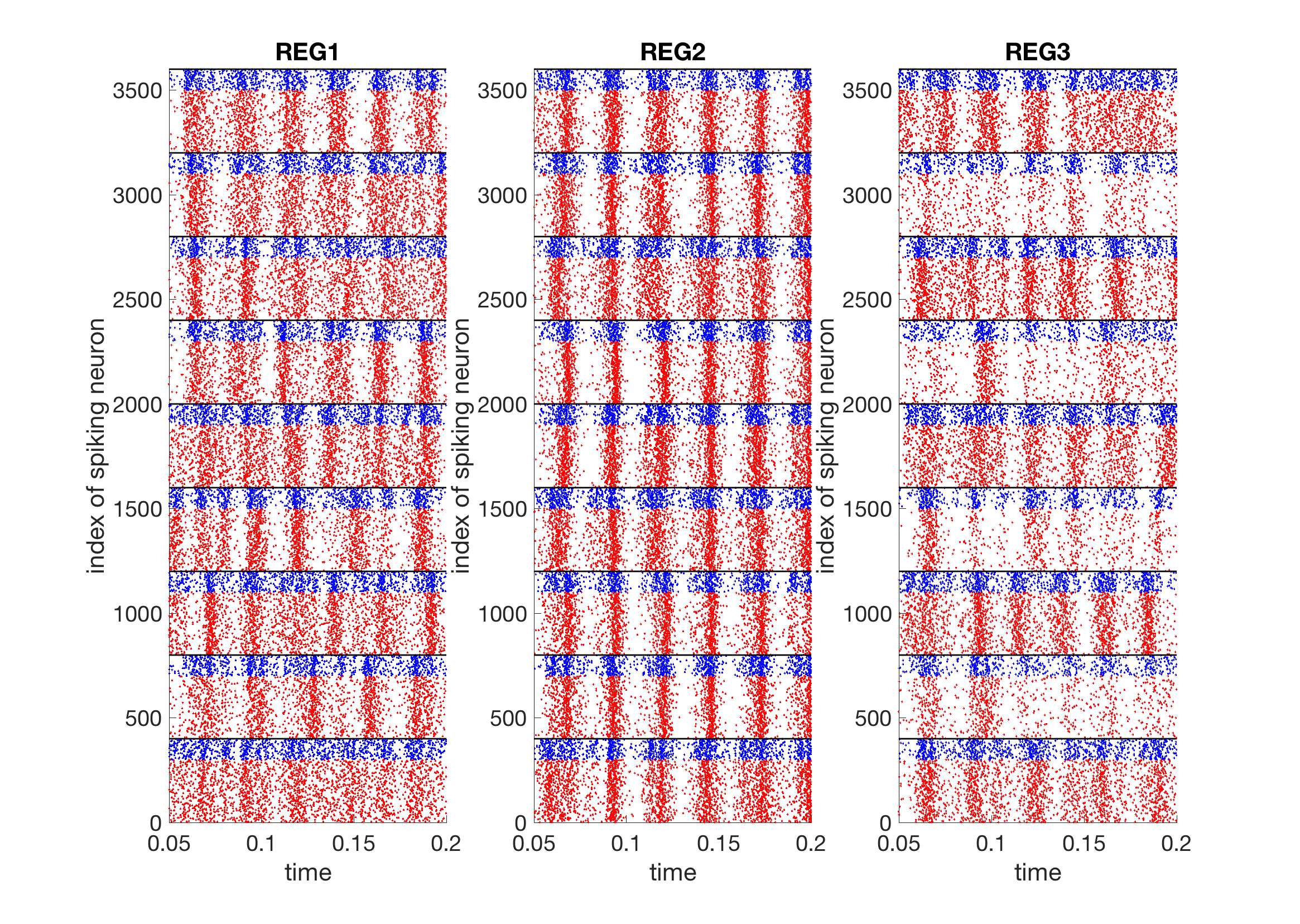}}
%\centerline{\includegraphics[width = \linewidth]{RegularPlot.png}}
\caption{Rastor plot of the full population of three {\bf REG}
  networks. The drive rate at even-indexed local populations
  is $\lambda_{Even} = 6000$ spikes/sec. The index of neuron $(m,n,k)$
  is $[(n-1)M + m](N_{E} + N_{I}) + k$.}
\label{reg}
\end{figure}

It remains to comment on the firing rate. The mean firing rate of the
central block is presented in Figure \ref{rate}. The drive rate of
even-indexed population varies from $\lambda_{Even} = 1000$ to
$\lambda_{Even} = 8000$. Different from our previous result in
\cite{li2017well}, the synchronized network {\bf SYN} now fires a lower
rate when the external drive is very strong. We believe the reason is
that inhibitory kicks in paper \cite{li2017well} are
voltage-dependent. As a result, when the spiking activity is very
synchronized, a neuron tends to receive lots of inhibitory kicks when
it just jumps out from state $\mathcal{R}$ (i.e., low membrane
potential). Hence the effective inhibitory current in the synchronized
network in \cite{li2017well} is weaker, which contributes to a higher
firing rate there. 
\begin{figure}[h]
\centerline{\includegraphics[width = \linewidth]{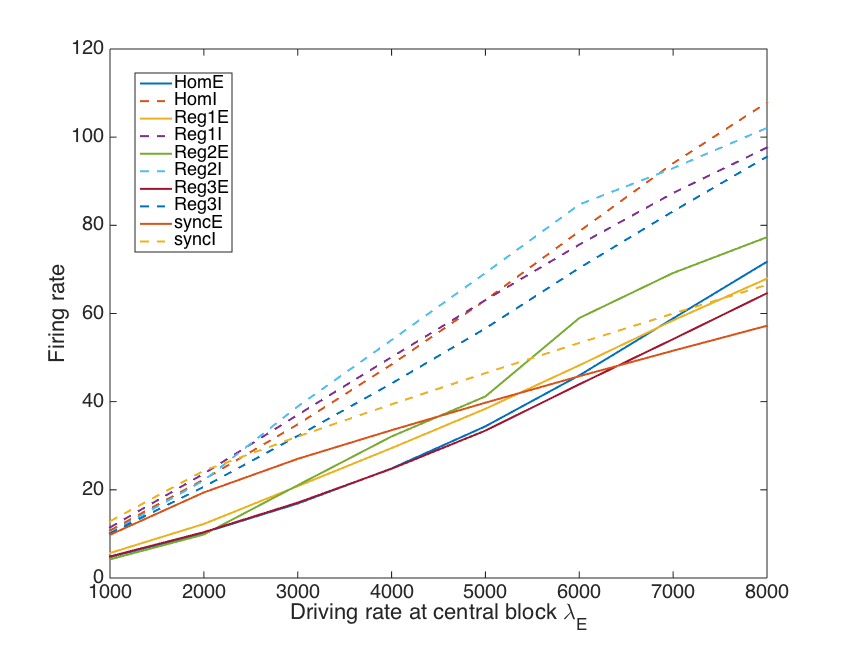}}
\caption{Mean firing rate at the central local population verses
  $\lambda_{Even}$ for all five networks. The drive rate at even-indexed local populations
  increases from $\lambda_{Even} = 1000$ spikes/sec to $8000$
  spikes/sec.}
\label{rate}
\end{figure}

\section{Comparing firing rates with mean-field approximations}
The aim of this section is to study the mean-field-type approximations
of the network model. Two reduced models with exactly solvable mean
firing rates are proposed in Section \ref{51} and \ref{52}. In Section
\ref{53}, we compare the mean firing rate produced by
these reduced models with the empirical firing rate of the
network model, and analyze the discrepancy between these firing rates.

\subsection{Reduced linear model}
\label{51}
Similarly as in the reduced models for a homogeneous population of
neurons studied in \cite{li2017well}, we assume that the membrane potential of each neuron $i$ changes at a constant speed and resets from $1$ to $0$ after firing without refractory state,
\begin{equation}
\label{drift}
\frac{dv}{dt} = F^+ - F^-, \quad \text{for $v \in$ [0,1]},
\end{equation}
where $F^+$ and $F^-$ are forces that drive membrane potential upward and downward respectively. In particular, with respect to the quantities defined previously, we have 
\begin{align}
& C_{EE} = N_E P_{EE} S_{EE},  \qquad  C_{IE} = N_E P_{IE} S_{IE},\\ \nonumber
& C_{EI} = N_I P_{EI} S_{EI}, \qquad C_{II} = N_I P_{II} S_{II},\\ \nonumber
& D_{EE} = N_E \rho_{EE} S_{EE},  \qquad  D_{IE} = N_E \rho_{IE} S_{IE},\\ \nonumber
& D_{EI} = N_I \rho_{EI} S_{EI}, \qquad D_{II} = N_I \rho_{II} S_{II}.
\end{align}
We can then define upward and downward drifting speed for excitatory neurons in local population $L_{m,n}$ as 
\begin{eqnarray}
\label{linear1}
 F^+ &=& \frac{1}{M} \left( f^E_{m,n}C_{EE} + \sum_{(m',n') \in \mathcal{N}(m,n)} f^E_{m',n'}D_{EE} +\lambda^E_{m,n} \right), \\ \nonumber
 F^- &=& \frac{1}{M} \left( f^I_{m,n} C_{EI}+ \sum_{(m',n') \in \mathcal{N}(m,n)} f^I_{m',n'}D_{EI}\right), 
\end{eqnarray}
and for inhibitory neurons in local population $L_{m,n}$ as 
\begin{align}
\label{linear2}
& F^+ = \frac{1}{M} \left(f^E_{m,n}C_{IE} + \sum_{(m',n') \in \mathcal{N}(m,n)} f^E_{m',n'} D_{IE} +\lambda^I_{m,n} \right) \\ \nonumber
& F^- = \frac{1}{M} \left( f^I_{m,n} C_{II} + \sum_{(m',n') \in \mathcal{N}(m,n)}
  f^I_{m',n'} D_{II} \right) \,.
\end{align}
As introduced in the previous chapter,
$f^E_{m,n}$ and $f^I_{m,n}$ are mean excitatory and inhibitory firing
rates for local population $L(m,n)$. Since we assume each population
$L_{m,n}$ to be homogenous, self-consistency results in an explicit
expression of firing rates, which can be derived in a similar way as
in \cite{li2017well}. To be precise, the above mentioned linear system can be expressed by the
form $\mathbf{A} \mathbf{f}=\mathbf{b}$, where $\mathbf{f}$ and $\mathbf{b}$ are  
\begin{equation}
\mathbf{f} = \begin{pmatrix} f^E_{1,1} \\ f^I_{1,1} \\ f^E_{1,2} \\
  f^I_{1,2} \\ \ldots \\ f^E_{m,n} \\ f^I_{m.n}\end{pmatrix}  \qquad
\mathbf{b} = - \begin{pmatrix} \lambda^E_{1,1} \\ \lambda^I_{1,1} \\  \lambda^E_{1,2} \\ \lambda^I_{1,2} \\ \ldots \\ \lambda^E_{m,n} \\ \lambda^I_{m.n}\end{pmatrix},
\end{equation}
and $\mathbf{A}$ can be derived using \eqref{linear1} to \eqref{linear2}. 

For example, if $M = N =2$, the coefficient matrix $\mathbf{A}$ is
\begin{equation}
 \begin{pmatrix} 
 C_{EE}-M & -C_{EI} & D_{EE} & -D_{EI} & D_{EE} & -D_{EI} & 0 & 0 \\
 C_{IE} & -C_{II}-M & D_{IE} & -D_{II} & D_{IE} & -D_{II} & 0 & 0 \\
 D_{EE} & -D_{EI} & C_{EE}-M & -C_{EI} & 0 & 0 & D_{EE} & -D_{EI}\\
 D_{IE} & -D_{II} & C_{IE} & -C_{II}-M & 0 & 0 & D_{IE} & -D_{II}\\
 D_{EE} & -D_{EI} & 0 & 0 & C_{EE}-M & -C_{EI} &  D_{EE} & -D_{EI} \\
 D_{IE} & -D_{II} & 0 & 0 &  C_{IE} & -C_{II}-M &  D_{IE} & -D_{II} \\
 0 & 0 & D_{EE} & -D_{EI} & D_{EE} & -D_{EI} &  C_{EE}-M & -C_{EI} \\
 0 & 0 &  D_{IE} & -D_{II} &  D_{IE} & -D_{II} & C_{IE} & -C_{II}-M 
 \end{pmatrix} 
\,.
\end{equation}
The explicit form of $\mathbf{A}$ gets complicated quickly with more
populations. But the solution of firing rates always exists whenever
the coefficient matrix $\mathbf{A}$ is invertible. By the perturbation theory
of matrices, it is easy to see that $\mathbf{A}$ is invertible if matrix
$$
 \begin{bmatrix}
C_{EE} - M & -C_{EI}\\
C_{IE} & -C_{II} - M
\end{bmatrix}
$$
is invertible and the coupling strengths $\rho_{Q_{1}Q_{2}}$ are
sufficiently small for $Q_{1}, Q_{2} \in \{E, I\}$.

\subsection{Reduced quadratic model}
\label{52}
A reduced quadratic model can be built
upon the reduced linear model in a similar way, with a slight improvement by including a
fixed refractory period after each firing event. Namely, the
normalized membrane potential satisfies the same drift condition
described in \eqref{drift}, except that whenever $V$ resets from $1$
to $0$, it stays at $0$ for a fixed amount of refractory time
$\tau_{\mathcal{R}}$ before resuming its linear climb.  

Using the self-consistency condition again, we now derive a system of quadratic equations for the excitatory and inhibitory firing rates $f^E_{m,n}$ and $f^I_{m,n}$ of population $L_{m,n}$ as follows,
\begin{align}
\label{quadratic}
& M f^E_{m,n} = \left(1-\tau_{\mathcal{R}} f^E_{m,n} \right) \\ \nonumber
& \hspace{0.7 in} \left(f^E_{m,n} C_{EE} + \sum_{(m',n') \in \mathcal{N}(m,n)} f^E_{m',n'} D_{EE} +\lambda^E_{m,n} -f^I_{m,n} C_{EI} - \sum_{(m',n') \in \mathcal{N}(m,n)} f^I_{m',n'} D_{EI} \right) \\  \nonumber
& M f^I_{m,n} = \left(1-\tau_{\mathcal{R}}f^I_{m,n} \right) \\ \nonumber 
& \hspace{0.7 in} \left(f^E_{m,n} C_{IE}+ \sum_{(m',n') \in \mathcal{N}(m,n)} f^E_{m',n'} D_{IE} +\lambda^I_{m,n} -f^I_{m,n} C_{II} -\sum_{(m',n') \in \mathcal{N}(m,n)} f^I_{m',n'}D_{II}\right).
\end{align}

\begin{lem}
Suppose that the reduced quadratic model has a unique solution
$\mathbf{f}$ when
$\tau_{\mathcal{R}}=0$. Then for sufficiently small
$\tau_{\mathcal{R}}> 0$, equation \eqref{quadratic} admits a solution
near $\mathbf{f}$. 
\end{lem}

\begin{proof}
Using notations from linear reduced model, we can write the above quadratic system as 
\begin{equation}
\mathbf{A}\mathbf{x}+ \tau_{\mathcal{R}} f(\mathbf{x}) = b,
\end{equation}
where $\tau_{\mathcal{R}}  f(\mathbf{x})$ corresponds to the small perturbation of quadratic terms. Assuming that $\mathbf{A}$ is invertible and the linear system has a solution, we can define a new function on $\mathbb{R}^{2MN}$ as 
\begin{equation*}
\label{g}
g(\mathbf{x}) = \mathbf{A}^{-1} \left( \mathbf{A}\mathbf{x} +
  \tau_{\mathcal{R}}  f(\mathbf{x}+ \mathbf{f}) -\mathbf{b} \right) +
\mathbf{A}^{-1}\mathbf{b} = \mathbf{x} + \tau_{\mathcal{R}}
\mathbf{A}^{-1} f(\mathbf{x}+ \mathbf{f}), 
\end{equation*}
where $\mathbf{f}$ is the solution when $\tau_{\mathcal{R}} = 0$. Notice that
\eqref{g} is the identity function when $f(\mathbf{x}) = 0$. Consider the
function $f$ within the hypercube $[-1,1]^{2MN}$, we have $\|f(\mathbf{x}+ \mathbf{f})\|
<c $ for some constant $c$ depending on
parameters. Therefore, for sufficiently small $\tau_{\mathcal{R}}$,
$\|\mathbf{A}^{-1} f(\mathbf{x}+ \mathbf{f}) \| <1$ for all $\mathbf{x} \in [-1,1]^{2MN}$. This means that
for all $\hat{\mathbf{x}}:=(x_1, \ldots, x_{i-1}, -1, x_i,\ldots, x_{2MN}),
\text{ and } \tilde{x} :=  (x_1, \ldots, x_{i-1}, 1, x_i,\ldots,
x_{2MN}) \in [-1,1]^{2MN}$, where $i \in \{1,2, \ldots, 2MN\}$, 
\begin{align*}
& g_i(\hat{\mathbf{x}}) = -1 + \tau_{\mathcal{R}} A^{-1}
  f(\hat{\mathbf{x}}+ \mathbf{f})(i) < 0,\\
& g_i(\tilde{\mathbf{x}}) = 1 + \tau_{\mathcal{R}}  A^{-1} f(\tilde{\mathbf{x}} + \mathbf{f})(i) > 0.
\end{align*}
By Poincare-Miranda theorem, which is a generalization of the intermediate value theorem, $g(\mathbf{x})$ has a zero $\mathbf{x}^*$ in the hypercube $[-1,1]^{2MN}$. Using the substitution $\mathbf{y} = \mathbf{x}^* + \mathbf{f}$, we can easily derive that $\mathbf{A}\mathbf{y}+\tau_{\mathcal{R}} f(\mathbf{y}) = \mathbf{b}$.
\end{proof} 
\subsection{Analysis and Comparison}
\label{53}
We compare the firing rate predictions of the reduced linear and
quadratic models against our stochastic network model in all five chosen
networks with varying degrees of synchronization. For the sake of
simpler notation, we denote the firing rate from stochastic model by $f^Q_{m,n}$, and
that from linear and quadratic models by $\tilde{f}^Q_{m,n}$ and
$\hat{f}^Q_{m,n}$ respectively, where $Q \in \{E,I\}$. We define the
mean relative errors of linear and quadratic predictions to be  
\begin{equation*}
\tilde{REL}^Q := \frac{1}{MN}\sum_{m,n} \frac{f^Q_{m,n} - \tilde{f}^Q_{m,n}}{f^Q_{m,n}}, \qquad \hat{REL}^Q := \frac{1}{MN}\sum_{m,n} \frac{f^Q_{m,n} - \hat{f}^Q_{m,n}}{f^Q_{m,n}},
\end{equation*}
respectively for $Q \in \{E,I\}$. Note that we do not take absolute
value because we would like to discuss the overestimate and
underestimate of the mean-field approximations later in this section. Figure \ref{compare} plots the
relative errors in all five networks in the sequence of {\bf HOM},
{\bf REG1},{\bf REG2}, {\bf REG3}, and {\bf SYN} from left to right. We have two observations
from Figure \ref{compare}: (i) the linear approximation
$\tilde{f}^Q_{m,n}$ is always smaller than the quadratic approximation
$\hat{f}^Q_{m,n}$, and (ii) the quadratic approximation tends to
underestimate the mean firing rate when the partial synchronization is weak
and overestimate when a strongly drived network is very synchronized. 

\begin{figure}[h]
\centerline{\includegraphics[width = 1.0 \linewidth]{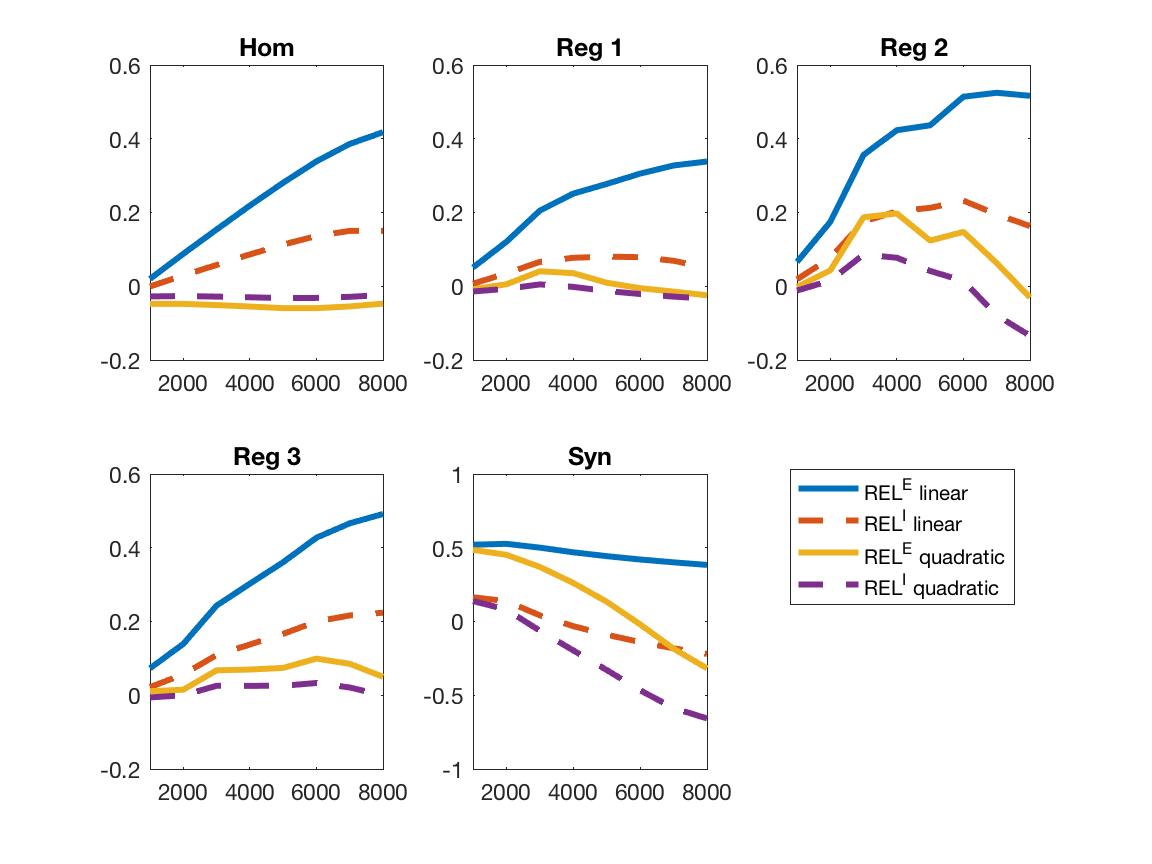}}
\caption{Relative errors of excitatory and inhibitory firing rate
  predictions for both linear and quadratic models in all five
  networks. The sequence of networks are HOM, REG1, REG2, REG3, and
  SYN from left to right. The driving rate increases from $\lambda =
  1000$ spikes/sec to $8000$ spikes/sec. } 
\label{compare} 
\end{figure}

The mechanism of the first observation is simple. Similarly as in our
previous paper \cite{li2017well}, the
inhibitory firing rate is always significantly higher than the excitatory
firing rate, leading to a larger fraction of inhibitory kicks being
missed during refractory than excitatory kicks. This causes the neuron
system to be more excited in the quadratic model than in the linear model
without refractory. 

It remains to discuss the discrepancy between the empirical mean
firing rate and the mean-field approximation in each network.  We
found that the reduced quadratic model gives decent approximation when
the network has weak synchronizations, i.e.,  examples {\bf HOM}, {\bf
  REG1}, and
{\bf REG3}. On the other hand, when the network becomes more synchronous, one observes significant
discrepancy between the network model and its mean-field approximation
(examples {\bf REG2} and {\bf SYN}). Different from the numerical result in
\cite{li2017well}, the reduced quadratic model overestimates the mean
firing rate when the network is in a strong synchronization. This can
be seen from the plot of {\bf SYN} network and {\bf REG2} network with strong
driving in Figure \ref{compare}.

We conclude that such discrepancy is caused by the
partial synchronous spiking activity during multiple firing
events. The derivation of both mean-field models relies on the assumption that the arrival
of postsynaptic kicks is homogeneous in time, which is clearly
violated during synchronous spiking activities. Right after a multiple
firing event, many neurons will stay
at $\mathcal{R}$ and be irresponsive to incoming postsynaptic kicks. As a result, a
disproportionately large fraction of postsynaptic kicks are missed in
a few ms after a large spiking volley. As demonstrated in Figure
\ref{missing}, the percentage of missed synaptic input is higher
than the percentage of time spend in refractory in all network
examples. This ``additional fraction'' of missing spike is not
negligible in network {\bf REG2} and significant in network {\bf SYN}.  

\begin{figure}[htbp]
\centerline{\includegraphics[width = \linewidth]{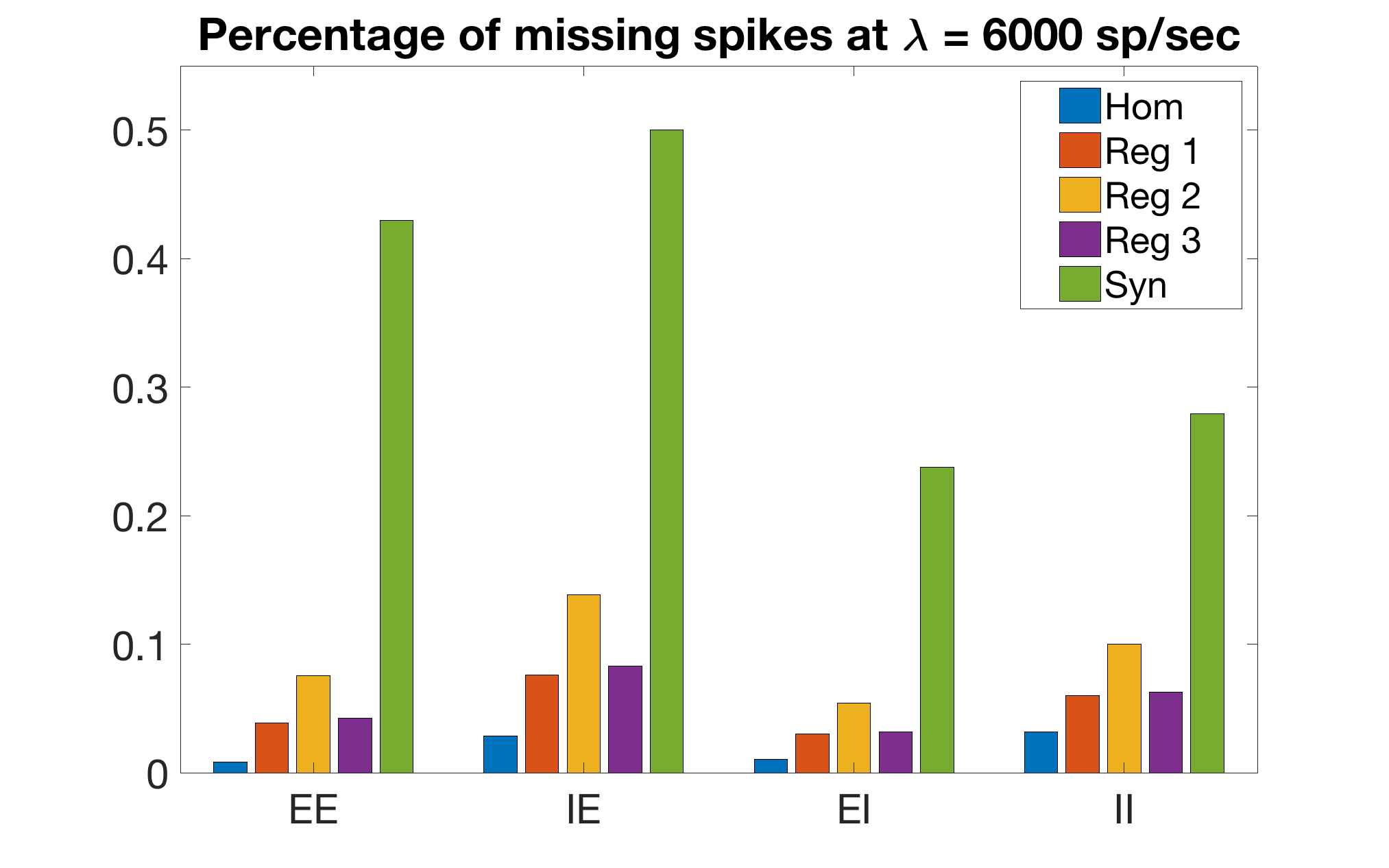}}
\caption{The ``additional fraction'' of missing E-E, E-I, I-E, and I-I
  spikes in all five
  network examples. ``Additional fraction'' means the percentage of
  missing $Q_{1}-Q_{2}$ spikes subtracts the average percentage of time that a
  postsynaptic neuron of the type $Q_{2}$ spends at the refractory,
  where $Q_{1}, Q_{2} \in \{ E, I\}$. All missing spike percentages are averaged
  over$9$ local populations.}
\label{missing}
\end{figure} 

For $Q_{1}, Q_{2} \in \{ E, I\}$, let $\epsilon^{Q_{1}Q_{2}}_{m,n}$ be the percentage of ``additional''
missing spikes, which means that the average percentage of time duration
for neuron $Q_{2}$ staying at the refractory is subtracted from the
empirical missing spikes proportions. Let $\Delta F^{E}_{m,n}$ be the
net gain of excitatory current as compared
with the reduced model. In the regime when the quadratic
approximation $\hat{f}^Q_{m,n}$ remains a good approximation of the
network firing rate $f^{Q}_{m,n}$, we have
\begin{eqnarray*}
  \Delta F^{E}_{m,n} &\approx& \left( \epsilon^{EI}_{m,n} f^{I}_{m,n}C_{EI} +
  \sum_{(m',n') \in \mathcal{N}(m,n)} \epsilon^{EI}_{m',n'} f^{I}_{m',n'}D_{EI} \right )\\
&& - \left ( \epsilon^{EE}_{m,n} f^{E}_{m,n}C_{EE} +
  \sum_{(m',n') \in \mathcal{N}(m,n)} \epsilon^{EE}_{m',n'} f^{E}_{m',n'}D_{EE}  \right) \,.
\end{eqnarray*}
Similarly, we have the net gain of
inhibitory current is given by
\begin{eqnarray*}
  \Delta F^{I}_{m,n} &\approx& \left( \epsilon^{II}_{m,n} f^{I}_{m,n}C_{II} +
  \sum_{(m',n') \in \mathcal{N}(m,n)} \epsilon^{II}_{m',n'} f^{I}_{m',n'}D_{II} \right )\\
&& - \left ( \epsilon^{IE}_{m,n} f^{E}_{m,n}C_{IE} +
  \sum_{(m',n') \in \mathcal{N}(m,n)} \epsilon^{IE}_{m',n'} f^{E}_{m',n'}D_{IE}  \right)
   \,. 
\end{eqnarray*}

Note that the missing percentage of excitatory postsynaptic kicks is usually
higher than that of the I kicks, partially due to the longer synapse
delay time of I-kicks. Put empirical missing percentages and constants in all five network examples into expressions $\Delta
F^{E}_{m,n}$ and $\Delta
F^{I}_{m,n}$. We can see that when $f_{I}^{m,n}$ is significantly
larger (at least $1.5$ times larger) than $f_{E}^{m,n}$, we have
positive net gain for both empirical $E$ and empirical $I$ current, corresponding to the
underestimate of the quadratic model.

However, when the network is very synchronous, all neurons spike in a
semi-periodic way, and $\hat{f}^Q_{m,n}$ can be very far away from
$f^Q_{m,n}$. In this regime, we find that the firing rate is simply
approximated by 
$$
  \bar{f}^{E}_{m,n} = \bar{f}^{I}_{m,n} \approx
  \frac{1}{\frac{M}{\lambda_{m,n}} + \tau_{\mathcal{R}}} \,,
$$
where $\lambda_{m,n} = \lambda^{E}_{m,n} = \lambda^{I}_{m,n}$. In
other words, in network {\bf SYN}, the firing activity is so synchronized
that most neurons participate in a multiple firing event, and restart
from the refractory right after it. The accuracy of this estimate is
presented in Figure \ref{periodic}, in which we calculate the mean
relative error $\overline{REL}^{Q}$ in the same way as calculating
$\tilde{REL}^Q$ and $\hat{REL}^Q$ and plot it at different drive
rates. We can see that $\overline{REL}^{Q}$ gives a better approximation of the mean firing rate of SYN
network, especially for excitatory local populations.  Comparing $\bar{f}^{Q}_{m,n}$ with
$\hat{f}^{Q}_{m,n}$, we find that the quadratic formula underestimates
the mean firing rate at low drive, and overestimates it at high
drive. This mechanism also partially explains the overestimate of
the quadratic formula in network {\bf REG2} at high drive rate, at which 
synchronization is also very significant.  

\begin{figure}[htbp]
\centerline{\includegraphics[width = 0.6\textwidth]{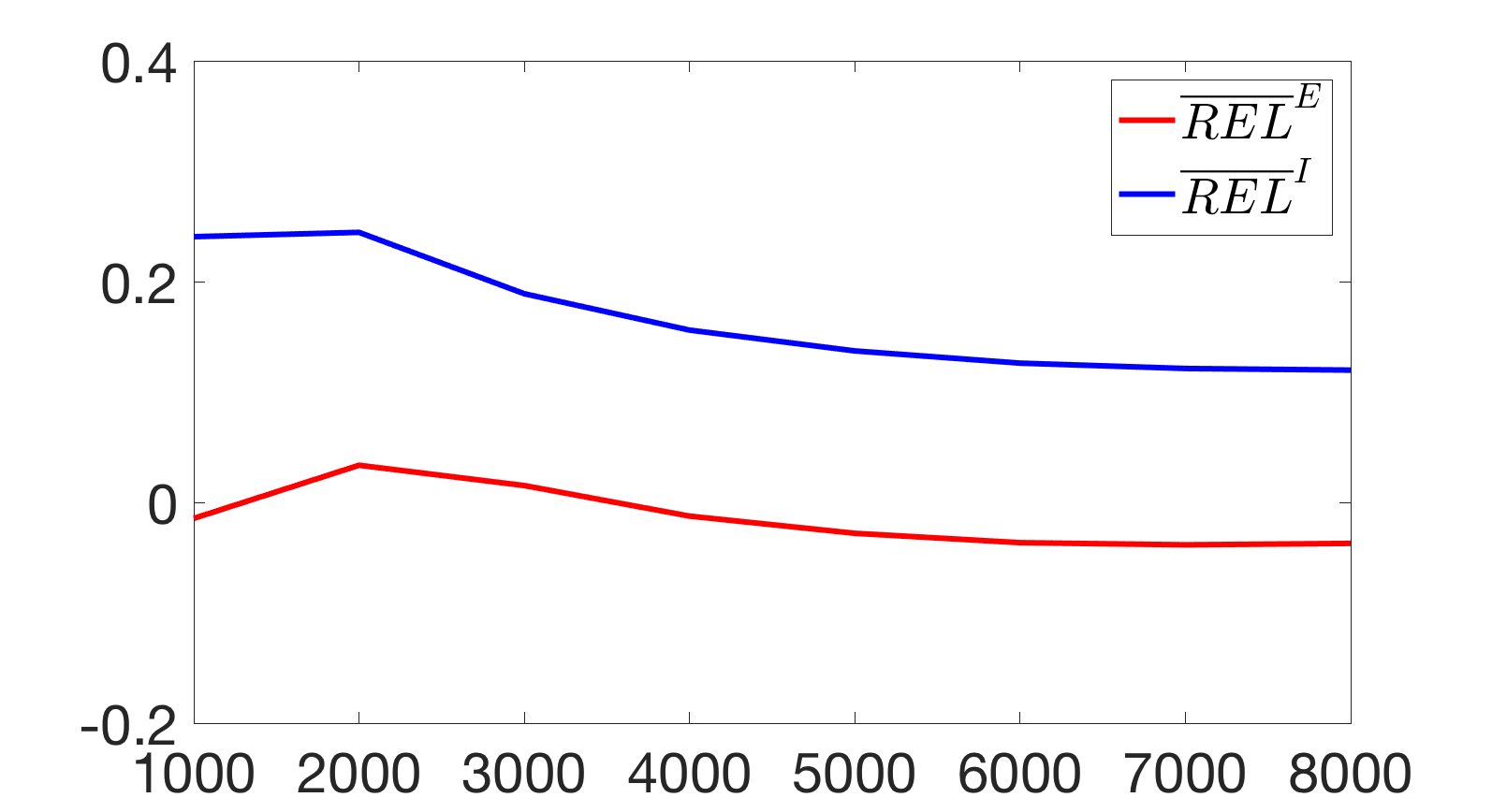}}
\caption{The total relative error $\overline{REL}^{E}$ and
  $\overline{REL}^{I}$ by assuming that \label{fig:3}the network is periodic.}
\label{periodic}
\end{figure}

Further numerical analysis shows that this $\bar{f}^{Q}_{m,n}$
actually has overestimate and underestimate that
significantly cancel each other. $\bar{f}^{Q}_{m,n}$ overestimates the
mean firing rate in a way that many pending inhibitory kicks can
survive after the refractory, and underestimates the mean firing rate
because a synchronous spiking event occurs when {\it some} membrane
potential reaches $M$, at which time the mean membrane potential is
still well below $M$. We note that such analysis for a very synchronized network is
not given in our previous paper \cite{li2017well}, as the
network {\bf SYN} in this paper is much more synchronous than examples there.

\section{Spatial correlation of spike volleys}
The semi-synchronized bursts of neuron spikes, i.e., the multiple firing events, observed in Figure
\ref{reg} are consistent with our experimental results and numerical
results \cite{chariker2015emergent,
  rangan2013dynamics, rangan2013emergent}. The scale of a multiple
firing event and the time interval between two consecutive events may vary. This is an
emergent phenomenon that is different from the total synchronization
studied in \cite{borgers2003synchronization, borgers2005effects, borgers2005background}. It is believed that such semi-synchronized burst is
related to the gamma rhythm in our brain \cite{borgers2005background,
  henrie2005lfp}, and is resulted by the interplay of excitatory and inhibitory
populations, which is a milder version of the PING mechanism \cite{borgers2003synchronization}. Inhibitory (GABAergic) synapses in a population usually
act a few milliseconds more slowly than excitatory (AMPA) synapses. As a
result, an excitatory spike will excite many postsynaptic neurons
quickly and form a cascade, which will be terminated when the pending
inhibitory kicks take effect. In \cite{li2017well}, we have found that the
degree of synchronization is extremely sensitive with respect to small
changes of the synapse delay times $\tau_{E}$ and $\tau_{I}$. However,
very limited mathematical justification is available so far.

One salient phenomenon observed in our numerical simulations is that
spike volleys generated by different local populations are
correlated. We call such correlated spiking activity among different
local populations the {\it spatial correlation}. One interesting
observation is that under ``reasonable''
parameter settings, this correlated spiking activity can only spread
to several blocks away. The aim of this
section is to investigate two questions: {\bf (i)} What is the mechanism of this
spatial correlation? and {\bf (ii)} How far away could
this spatial correlation spread? We will describe our numerical
results about spatial correlation in Section \ref{6-1} and
\ref{6-2}. The mechanism of spatial correlation will be studied in
Section \ref{6-3}. A study of the mechanism of correlation decay is
provided in Section \ref{6-4}. 

\subsection{Quantifying spatial correlations}
\label{6-1}

The quantification of spatial correlation relies on the ergodicity of
the Markov process. It follows from Corollary \ref{cor2} that for any two local populations $(m,n)$ and $(m', n')$
and any $Q_{1}, Q_{2} \in \{E, I\}$,
we have well-defined and computable covariance $\mbox{cov}_{T}^{Q_{1}, Q_{2}}(m,n;
m', n')$ and Pearson' correlation coefficient $\rho^{Q_{1},
  Q_{2}}_{T}(m,n,m',n')$. By Corollary \ref{cor3}, the covariance $\mbox{cov}_{T}(m,n,m',n')$ and correlation coefficient
$\rho_{T}(m,n,m',n') $ for the total spike count between two local
population (regardless the spike type) are also well defined and
computable. 

We remark that sometimes it makes sense to have different ``resolutions'' at different spike
counts. For example, whether a local population fires $5$ spike or $10$
spikes during a $10$ ms time window makes qualitative difference because
we may think $10$ spikes in such a time window gives a multiple firing
event. But whether a local population fires $295$ or $300$ spikes in
the same time window is less important. To address this, we can
prescribe a mapping on the spike count during $[0, T]$. Let 
$\bm{\xi} = \{\xi_{1}, \cdots, \xi_{k}\} \subset \mathbb{N}_{+}$ be a
``dictionary''. Let $f_{\bm{\xi}}: \mathbb{N_{+}} \mapsto \{1, \cdots,
k+1\}$ be a function such that 
$$
  f_{\bm{\xi}}(n) = \left \{
\begin{array}{ccc}
 \max \{ 1 \leq i \leq k \,|\,  n \leq \xi_{i} \} & \mbox{ if } & n \leq
                                                             \xi_{k}\\
k+1 & \mbox{ if } & n > \xi_{k}
\end{array}
\right .
$$
When starting from the steady state $\pi$, we have random
variables $f_{\bm{\xi}}(N^{Q}_{m,n}([0, T]))$ representing the mapping
of the spike count on $[0, T]$. The we can define the covariance
$\mbox{cov}_{T, \bm{\xi}}^{Q_{1},
      Q_{2}}(m,n,m',n')$ and the Pearson correlation coefficient
    $\rho^{Q_{1}, Q_{2}}_{\bm{\xi}, T}(m,n,m',n')$ in an
    analogous way. One advantage of using $f_{\bm{\xi}}$ to define the
    correlation is that $|\bm{\xi}|$ can be much smaller than $N_{E}$
    and $N_{I}$. Hence the estimates of covariance and correlation
    coefficient can be more accurate.

It remains to comment on the size of a time window $T$. An ideal time window
should be large enough to contain a spike volley, but not as large as
the time between two consecutive time spike volleys. There is no silver
lining of time-window size that fits all parameter sets. A (very) rough estimate
is that $T$ should be greater than the maximum
of $N_{Q}$ i.i.d exponential random variables with mean $\tau_{Q}$,
but less than $M/\lambda^{Q}_{m,n} + \tau_{\mathcal{R}}$. It is well known that the maximum
of $N_{Q}$ i.i.d exponential random variables with mean $\tau_{Q}$,
denoted by $Z$, can
be represented as the sum of independent exponential random variables
$$
  Z = W_{1} + \cdots + W_{N_{Q}} \,,
$$
where $W_k$ has mean $\frac{\tau_Q}{k}$. Hence the expectation
of $Z$ is $\tau_Q H_{N_Q}$, where $\{H_{n}\}$ is the Harmonic number
$$
  H_{n} = \sum_{k = 1}^{n} \frac{1}{k} \,.
$$
It is clear that $H_{N_{Q}} \tau_{Q}$  overestimates because a spike
volley may not involve all neurons. At the same time, it also underestimates because it
takes some time for the cascade of excitatory spikes to excite all
neurons participating in a spike volley. But our simulations shows
that the qualitative properties of the spatial correlation is not very
sensitive with respect to the choice of time-window size. For the
parameters of a ``{\bf regular}'' network, we have $H_{300}\tau_{EE}
\approx 10$ ms. Also we have $M/\lambda_{E} +
\tau_{\mathcal{R}}\approx 19$ ms if $\lambda_{E} = 6000$ (strong
drive). Hence we choose $T = 15$ ms in our simulations in the next
subsection.

\subsection{Spatial correlation decay}
\label{6-2}

Our first key observation is that in many settings, the spatial
correlation decays quickly when two local populations are further
apart. The aim of this subsection is to describe this numerical
finding. We will address possible mechanisms of spatial correlation
and spatial correlation decay in the following two subsections. 

As discussed in the last subsection, we choose $T = 15$ ms as the size
of a time window. Since the qualitative result for the spatial
correlation among E-E, E-I .etc are the same, we select
$\rho_{T}(m,n,m',n')$ as the metrics of the spatial correlation. The
cases of $\rho^{E,E}_{T}(m,n,m',n')$ have little difference. 

In order to effectively simulate large scale neural fields in which two local populations can
be far apart, we choose to study an $1$-D network with $M = 1$. The
length of array is chosen to be $N = 22$ in all of our simulations. We
will compare the Pearson correlation coefficient between local
populations $L_{1,2}$ and $L_{1,k}$ for $k = 2, \cdots, 21$. The
reason of doing this is to exclude the boundary effect at $L_{1,1}$
and $L_{1, N}$. In our simulation, we run $240$ independent long-trajectories of
$\Phi_{t}$. In each trajectory, spike counts in $2000$ time windows
are collected after the process is stabilized. The result of this
simulation is presented in Figure \ref{cordecay}. We find that in all five
example networks, the spike count correlation decays quickly with increasing
distance between local populations. In homogeneous network {\bf HOM},
correlation is only observed for nearest neighbor local populations. In three regular networks {\bf REG1, REG2}, and {\bf REG3}, no significant correlations are
observed when two local populations are $4-8$ blocks away. The speed
of correlation decay is higher when external drive rates have higher
difference ({\bf REG1}) and when the external connection is weaker ({\bf
  REG3}). The synchronized network {\bf SYN} has the slowest decay rate and obvious
fluctuations induced by alternating external drive rate at local populations. If one
local population models a hypercolumn in the visual cortex, our
simulation suggests that the Gamma wave does not have significant
correlation at two locations that are $2-3$ millimeters away. This is
consistent with experimental observations.

\begin{figure}[htbp]
\centerline{\includegraphics[width = \linewidth]{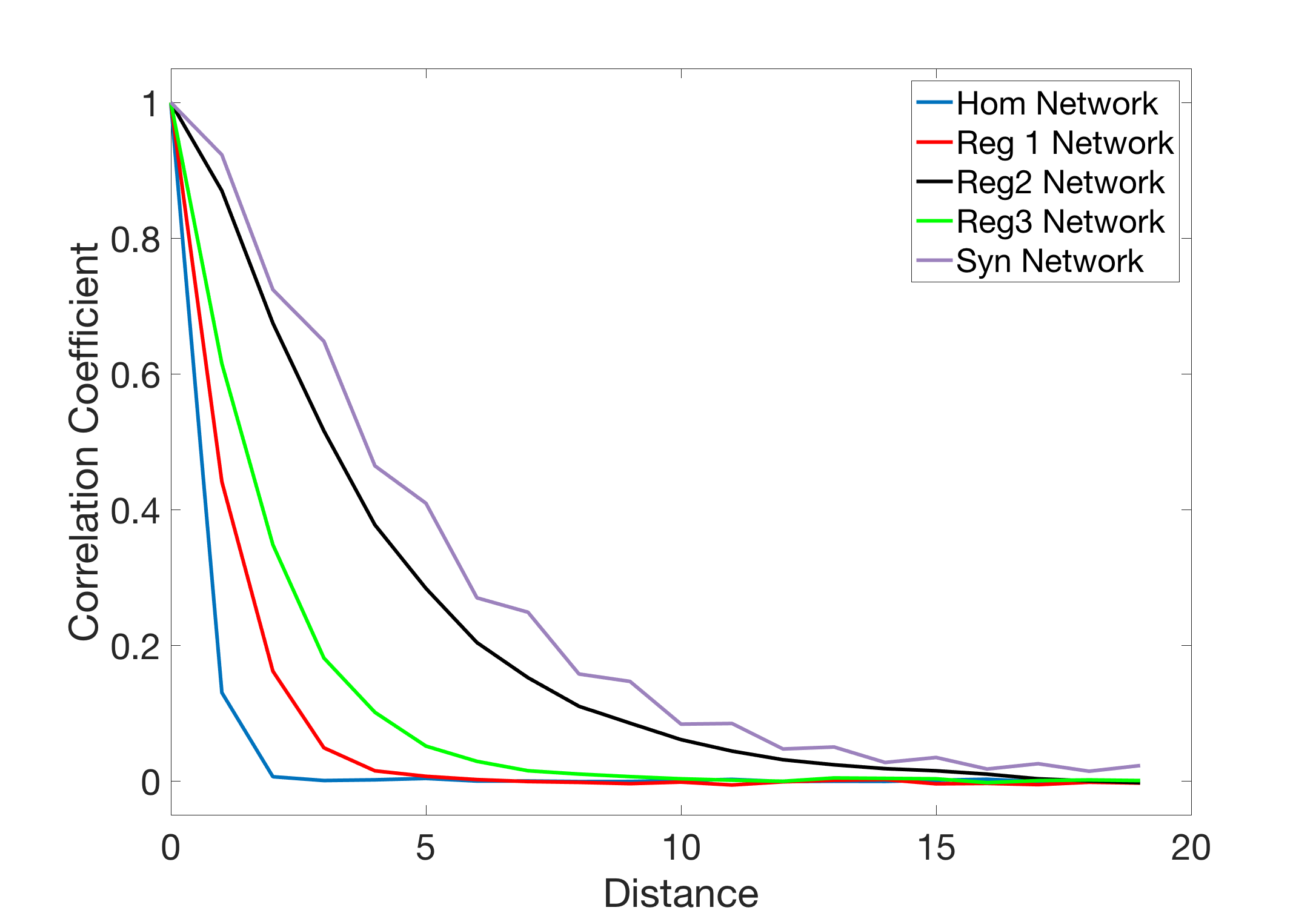}}
\caption{Change of spike count correlation coefficient with increasing
distance between to local populations. Correlation coefficients are
measured between $L_{1,2}$ and $L_{1, k}$ for $k = 2, \cdots, 21$ in
five example networks.}
\label{cordecay} 
\end{figure}

\subsection{Mechanism of spatial correlation}
\label{6-3}
The aim of this section is to investigate the mechanism of spatial
correlation, especially the spatial correlation of spike counts between two nearest
local populations. We believe that the mechanism of spatial correlation is
similar to that of a multiple firing event in a local population. The
excitatory neurons stimulate each other and form an avalanche of
spikes, which is terminated by the later arrival of inhibitory
kicks. Because both excitatory and inhibitory neurons connect to
nearest neighbors, two neighbor local populations tend to have spike
volleys at the same time. 

To better explain this dynamics, we first propose the following $6$-variable
ODE system
for a qualitative description of what happens during a spike
volley at one local population. The main changing parameters in our ODE model are $\tau_{E}$ and
$\tau_{I}$. Other parameters like $N_{E}$, $N_{I}$, $S_{EE}$, .etc are
same as in the model description. This system contains six variables
$G_{E}, G_{I}, H_{E}, H_{I}, R_{E}$, and $R_{I}$. Variables $G_{E}$
and $G_{I}$ are the number of excitatory and inhibitory neurons that
are located in the ``gate area'', which means their
membrane potentials lie within one excitatory spike from the
threshold. We denote the set of neurons in this ``gate area'' by
$G_{E}$ and $G_{I}$ when it does not lead to confusions. Variables
$H_{E}$ and $H_{I}$ are the ``effective'' number of excitatory and inhibitory
neurons who just spiked but the spikes have not taken effects
yet. If, for example, $50\%$ of postsynaptic kicks from a neuron spike have
already taken effects, the ``effective number'' of this neuron is
$0.5$. Finally, $R_{E}$ and $R_{I}$ are the number of excitatory and
inhibitory neurons that are at refractory. Since we only study the
dynamics of one spike volley, we assume that a neuron stays at
$\mathcal{R}$ after a spike. Let $c_{E}, c_{I}$ be two
parameters that will be described later, we have the following
multiple firing event model that describes the
time evolution of $H_{E}, H_{I}, G_{E}, G_{I}, R_{E}$, and $R_{I}$. 

\begin{equation}
\label{MFE}
\begin{split}
\frac{\mathrm{d} H_{E}}{\mathrm{d} t} = & - \tau_{E}^{-1} H_{E} +
                                           \tau_{E}^{-1}P_{EE}H_{E}G_{E}
  + \frac{\lambda_{E}}{S_{EE}} G_{E}\\
\frac{\mathrm{d} G_{E}}{\mathrm{d}t} =& c_{E} \max\{
                                         \tau_{E}^{-1}P_{EE}S_{EE}H_{E}
                                         + \lambda_{E}
                                         - \tau_{I}^{-1}P_{EI}S_{EI}H_{I},
                                         0\}\cdot ( N_{E} -
                                         G_{E} - R_{E}) -\\
&
                                         \tau_{I}^{-1}P_{EI}\max\{\frac{S_{EI}}{S_{EE}},
                                         1\}H_{I}G_{E} -
   (\tau_{E}^{-1}P_{EE}H_{E} + \frac{\lambda_{E}}{S_{EE}})G_{E}
                                        \\ 
\frac{\mathrm{d} R_{E}}{\mathrm{d}t}=&
                                        \tau_{E}^{-1}P_{EE}H_{E}G_{E}
                                        +  \frac{\lambda_{E}}{S_{EE}} G_{E}\\
\frac{\mathrm{d} H_{I}}{\mathrm{d} t} = & - \tau_{I}^{-1} H_{I} +
                                           \tau_{E}^{-1}P_{IE}H_{E}G_{I}
  + \frac{ \lambda_{I}}{S_{IE}}G_{I}\\
\frac{\mathrm{d} G_{I}}{\mathrm{d}t} =& c_{I} \max\{
                                         \tau_{E}^{-1}P_{IE}S_{IE}H_{E}
                                         + \lambda_{I}
                                         -
                                         \tau_{I}^{-1}P_{II}S_{II}H_{I},
                                         0\}\cdot ( N_{I} -
                                         G_{I} - R_{I}) \\
& -\tau_{I}^{-1}P_{II}\max\{ \frac{S_{II}}{S_{IE}}, 1\}H_{I}G_{I} -
   (\tau_{E}^{-1}P_{IE}H_{E} + \frac{\lambda_{I}}{S_{IE}})G_{I}\\
\frac{\mathrm{d} R_{I}}{\mathrm{d}t}=&
                                        \tau_{E}^{-1}P_{IE}H_{E}G_{I}
                                        + \frac{
                                          \lambda_{I}}{S_{IE}}G_{I}
\end{split}
\end{equation}

We note that the aim of this ODE system is to qualitatively describe the
excitatory-inhibitory interplay, instead of making any precise
predictions. The first equation describes the rate of change of
$H_{E}$, which decreases with rate $\tau_{E}^{-1}$. The source of input to
$H_{E}$ is neurons in the ``gate area'' $G_{E}$. We assume that neurons in
$G_{E}$ have uniformly distributed membrane potentials. The second
equation describes the rate of change of neurons in $G_{E}$. The
source of $G_{E}$ is neurons that are not at $G_{E}$ or $R_{E}$. We
assume that the increase rate of $G_{E}$ is proportional to both the number of
relevant neurons and the net current, if the net current is positive. The
coefficient is assumed to be a parameter $c_{E}$. We denote the coefficient of proportion as a parameter $c_E$. $G_{E}$ decreases because neurons in $G_{E}$ either spike
when receiving excitatory input or drop below the ``gate area'' when
receiving inhibitory input. This is represented by the last two terms
in the second equation. The third equation is about $R_{E}$, whose
increase rate equals the rate of producing new spikes. The case of
inhibitory neurons is analogous, represented by the last three equations about $H_{I}, G_{I}, R_{I}$, where parameter $c_I$ stands for the coefficient of net current for inhibitory neurons.

The most salient feature of this ODE system is the very sensitive dependency of
``event size'' with respect to $\tau_{E}$ and $\tau_{I}$. When
$\tau_{I}$ is much larger than $\tau_{E}$, one can expect larger
event sizes for both populations. This is demonstrated in Figure
\ref{bk1}. We assume that $\tau_{E} = 2$ ms and plot the ``event sizes''
with varying $\tau_{I}$. The initial condition is $H_{E} = 0, G_{E} =
20, R_{E} = 0, H_{I} = 0, G_{I} = 5$, and $R_{I} = 0$. We showed three
cases with varying drive rates, where $\lambda_{E} = \lambda_{I} = 0, 2000$, and $4000$. The
``event sizes'' of excitatory and inhibitory populations are $R_{E}(T)$
and $R_{I}(T)$ respectively, where $T$ is the minimum of $20$ ms and
the first local minimum of $H_{E}(t)$. The reason of looking for the
local minimum is because when the network is
driven, this ODE model might generate a ``second wave'' after the first
multiple firing event. Figure \ref{bk1} confirms two observations in
our simulation results. First, the ``event sizes'' of both excitatory
and inhibitory populations increase quickly with larger
$\tau_{I}$. Second, the network tends to have bigger multiple firing
events when it is strongly driven by external signals. 

\begin{figure}[htbp]
\centerline{\includegraphics[width = \linewidth]{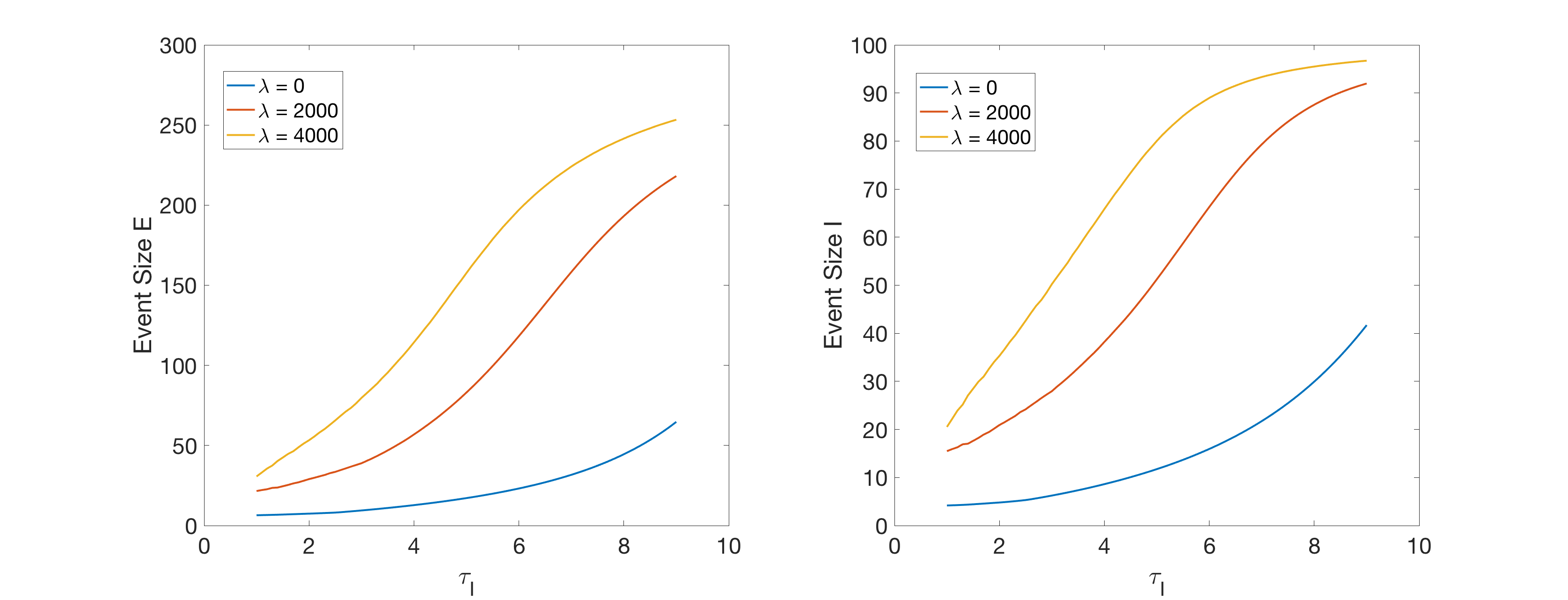}}
\caption{Event size versus $\tau_{I}$ for one local population.}
\label{bk1}
\end{figure}
In the limiting senario when $\tau_{I}^{-1}$ is very small, we have the
following theorem. 

\begin{thm}
\label{limitcase}
Assume $\tau_{E} = 1$, $\lambda_{E} = \lambda_{I} = 0$. Let
$\delta_{E} = c_{E}S_{EE}$, $\delta_{I} = c_{E}S_{IE}$, $m = \min\{\delta_{E}, P_{EE}\}$, 
$$
  \alpha = P_{EE}N_{E}^{2}e^{-N_{E}m} \,,
$$
and
$$
  \beta = N_{I}(e^{-\delta_{I}N_{E}} +
  e^{-P_{IE}N_{E}}) \,.
$$
Assume $N_{E} > m$. Let the initial condition be $(H_{0}, 0, 0, 0, 0, 0)$ for $H_{0} >
\alpha$. Then there exist constants $C$ and $T$, such that when
$\tau_{I} > C$ and $t > T$, we have $R_{E}(t) > N_{E} - \alpha$ and
$R_{I}(t) > N_{I} - \beta$.
\end{thm}
\medskip

Theorem \ref{limitcase} implies that as long as $\tau_{I}$ is sufficiently
large, even if without external drive, most
neurons will eventually spike provided that there are enough pending excitatory
spikes in the beginning. Note that $m$ is typically not a very small
number ($0.1$ in our simulations). Hence $\alpha$ and $\beta$ are both
very small numbers. 

The proof of Theorem \ref{limitcase} only contains elementary calculations, and we include it in Appendix A.

When many local populations form a $M\times N$ array, an ODE system with
$6MN$ variables can be
derived from the same approach. We include this ODE system and its
description in Appendix B. A direct analysis is too complicated to
be interesting. But the numerical result reveals the mechanism of
spatial correlation. For the sake of simplicity we consider two local
populations, say populations $L_{1,1}$ and $L_{1,2}$. Assume that the
local population $L_{1,1}$ is ready for a multiple firing event with
initial condition $H_{E} = 1, G_{E} = 30, R_{E} = 0, H_{I} = 0, G_{I}
= 10, R_{I} = 0$, while $L_{1,2}$ has a very different profile with $H_{E} = 0, G_{E} = 10, R_{E} = 0, H_{I} = 0, G_{I}
= 2, R_{I} = 0$. We further assume that $\tau_{E} = 2$ ms and $\lambda_{E} =
\lambda_{I} = 0$. It is easy to see that without $L_{1,1}$, $L_{1,2}$ will
not have any spikes because it is not driven. We compare the ``event size'' of excitatory and
inhibitory populations at $L_{1,1}$ and $L_{1,2}$, which is measured at
time $T = 20$ ms. This is demonstrated in Figure \ref{bk2}. With
strong connectivity $\mbox{ratio}_{E} = 0.15$, the multiple firing
event at local population $L_{1,1}$ will induce a multiple firing event
at local population $L_{1,2}$, even if local population $L_{1,2}$ has much fewer neurons at
the ``gate area''.

\begin{figure}[htbp]
\centerline{\includegraphics[width = \linewidth]{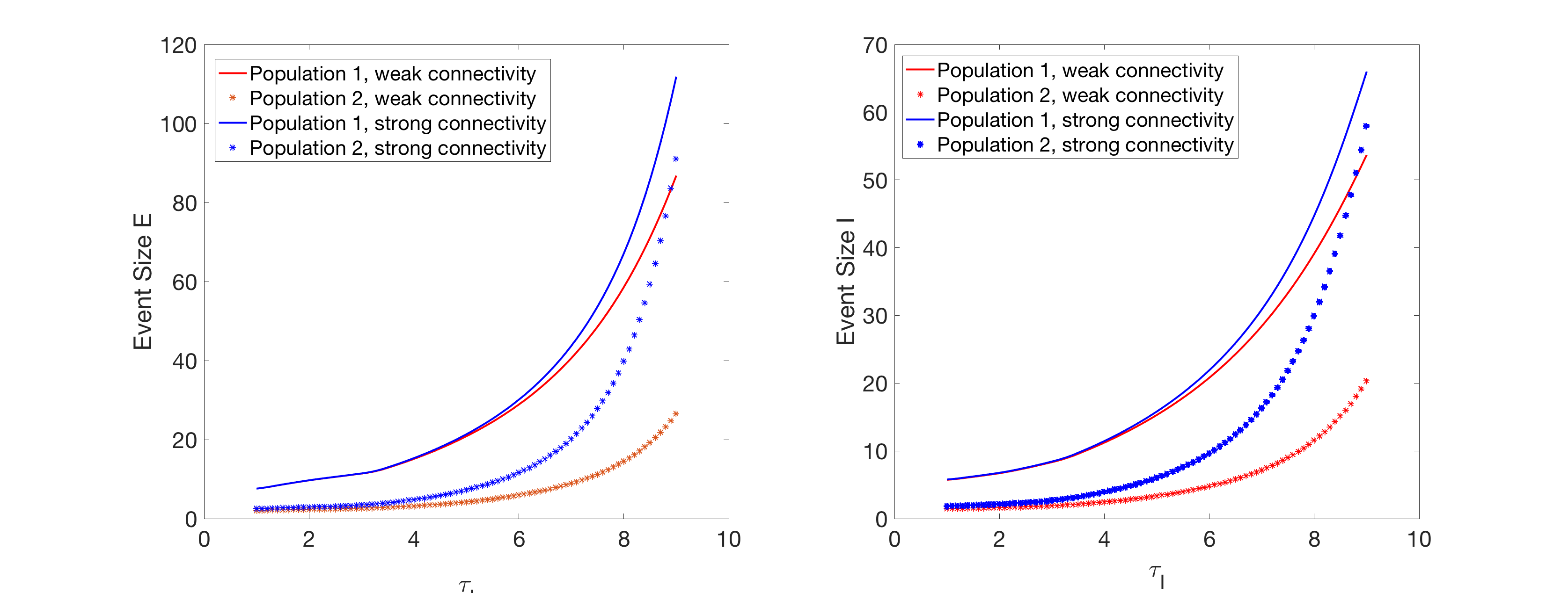}}
\caption{Event size versus $\tau_{I}$ for two local populations.}
\label{bk2}
\end{figure}

We believe that this is the mechanism of spike count spatial correlation in
our model. When a multiple firing event occurs in one local population, it
sends excitatory and inhibitory input to its neighboring local
populations. If the membrane potential of a neighboring local
population is properly distributed, a multiple firing event will be induced by
the activity at its neighbor. Similar to the case of one local
population, the size of a multiple firing event sensitively depends on the
excitatory and inhibitory synapse delay times.

\subsection{Mechanism of spatial correlation decay}
\label{6-4}
In the previous section we have shown that a strong multiple firing
event in a local population is very likely to induce a multiple firing
event at its neighboring local populations. This partially explains the
mechanism of the spatial correlation of spike volley that we have
observed. Our final task is to investigate the mechanism of spatial correlation
decay, as described in Figure \ref{cordecay}, where the spatial correlation of spike volley can only
spread to several local populations away. 

Why the spatial correlation can not spread to very far away? We
believe that (at least in this
model) such correlation decay is due to the volatility of spike count in a
multiple firing event. Since each neuron finds its postsynaptic neurons
in a random way, the spike count in a local population usually has
large variance. The variance will be even larger if the external drive rate
is heterogeneous. Therefore, even if the initial distributions of the
membrane potential and the external drive rates are identical
throughout all local populations, the voltage distribution will be
very different after the first multiple firing event. Therefore, the next
spike volley in different local populations will be less
coordinated, which destroys the synchronization. We believe this high volatility of spike volley size
significantly contributes to the spatial correlation decay. As shown
in Figure \ref{cordecay}, in examples {\bf REG1} and
{\bf REG2}, $\lambda^{E}_{m,n}$ and
$\lambda^{I}_{m,n}$ have very small difference in different local
populations. But the spatial correlation decay is still strong. This
also explains why the {\bf SYN} network has the weakest spatial
correlation decay. When the size of a spike volley is closer to the
size of the entire population, there will be much less variation in
the after-event voltage distribution. Hence the voltage distributions
in different local populations are relatively similar in a {\bf SYN}
network, which contributes to the observed slow decay of spatial
correlation. 

This explanation is supported by both analytical calculation and
numerical simulation result. We did the following numerical simulation
to investigate the volatility of ``event size''. Assume $M = N = 1$,
$\lambda_{E} = \lambda_{I} = 3000$, and the initial voltage
distribution is generated in the following way: With probability $0.2$, the neuron
membrane is uniformly distributed on $\{0, 1, \cdots, 80 \}$. With probability
$0.8$, the neuron membrane potential takes the integer part of a
normal random variable with mean $0$ and standard deviation $20$. This
initial voltage distribution roughly mimics the voltage distribution
after a large spike volley. The synapse delay times are $\tau_{E} = 2$
ms and $\tau_{I} = 1 \sim 9$ ms. For each $\tau_{I} = 1.0, 1.1,
\cdots, 9.0$, we simulate this model repeatedly for $10000$ times and
count the number of excitatory spikes of the first multiple firing
event. Then we plot the mean event size and the coefficient of
variation (standard deviation divided by mean) of the spike count
samples for each $\tau_{I}$. Note that the coefficient of
variation is a better metrics than the standard deviation as it is a
dimensionless number that measures the relative volatility of a
multiple firing event. 

This numerical result is shown in Figure \ref{CV}. We can see that when $\tau_{I}$ become larger, the event
size increases and the coefficient of variation decreases. The reason
of decreasing is because the size of a multiple firing event usually
can not be much larger than the number of neurons. Only a small number
of neurons have the chance to spike twice in a multiple firing event, even if in the
most synchronized network. The change of coefficient of variation
partially explains the observation in Figure \ref{cordecay}, in which the
decay of spatial correlation is slower when the
$\tau_{I}$-to-$\tau_{E}$ ratio is larger (means the network is more
synchronized). 

\begin{figure}[htbp]
\centerline{\includegraphics[width = \linewidth]{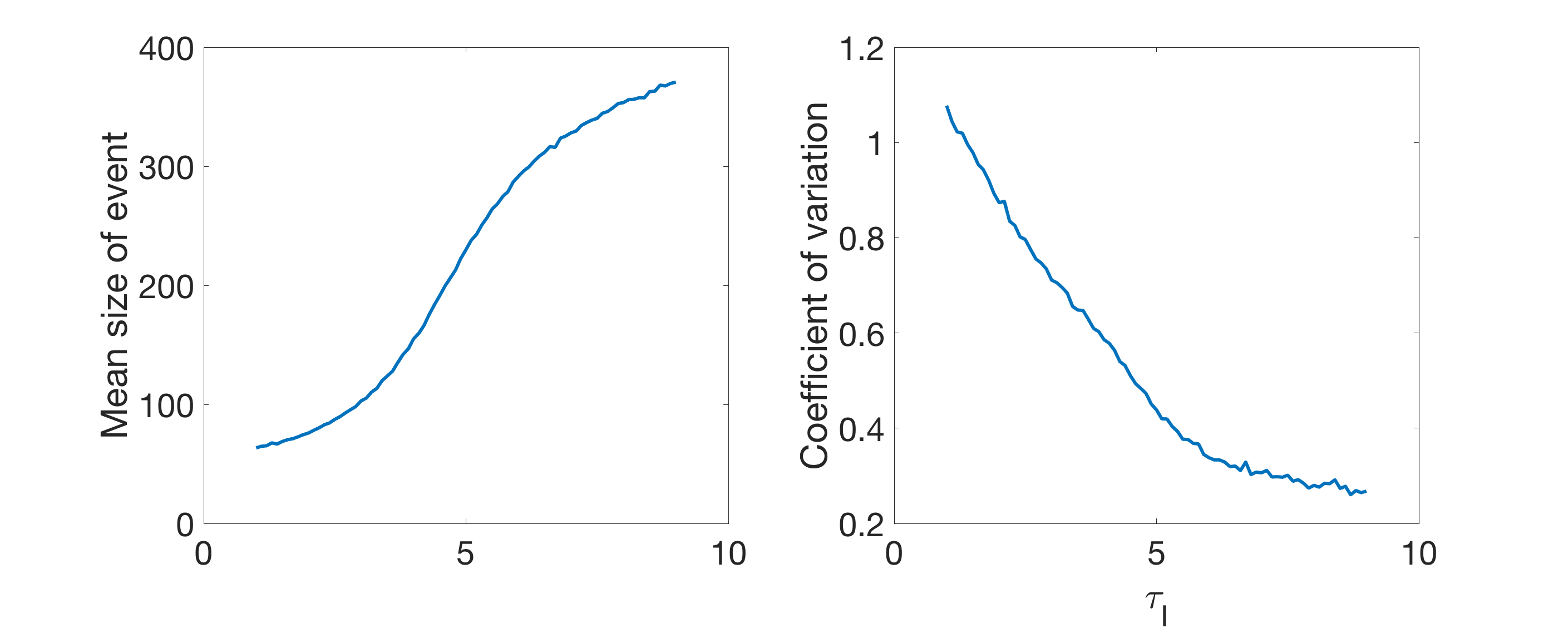}}
\caption{Left: Mean event size versus $\tau_{I}$. Right: Coefficient
  of variation versus $\tau_{I}$.}
\label{CV}
\end{figure}

It is difficult to directly study the event size on the population
model. But it is easy to build a simplified model to study the
mechanism of the high variance of multiple firing events. As explained
before, a multiple firing event is produced by the recurrent
excitation and the slower onset of the inhibition. Therefore, we
consider modeling the multiple firing event by stopping a Galton-Watson branching
process at a random time. For the sake of simplicity, we only consider
the case of $E$ population. Some calculation for a Galton-Watson
process will qualitatively explain the reason of high coefficient of
variation of sizes of multiple firing events. 

Let $X_{i} \sim B(N_{E}, p)$ be i.i.d Binomial random variables that
represent the numbers of new excitatory spike stimulated by an
excitatory spike. Let $Z_{n}$ be a branching process such that $Z_{0}
= 1$ and 
$$
  Z_{n+1} = X_{1} + \cdots + X_{Z_{n}} \,.
$$
Further let the total number of spikes before step $n$ be $S_{n} =
Z_{1} + \cdots + Z_{n}$.

Let $T$ be a positive integer-valued random variable that is
independent of all $X_{i}$ and $Z_{i}$. $T$ is the approximate onset time of
network inhibition. We further assume that $T$ has mean $\mu_{T}$ and
variance $\sigma_{T}^{2}$. The approximate event size is then $S_{T}$.
$$
  S_{T} = Z_{1} + \cdots + Z_{T} \,.
$$ 
Let 
$$
  \mbox{CV}(X) = \frac{\sqrt{\mathbb{E}[(X - \mathbb{E}[X])^{2}]}}{\mathbb{E}[X]}
$$
be the coefficient of variation of a positive-valued random variable
$X$. The following proposition is straightforward.

\begin{pro}
\label{cvbound}
Assume $T$ has finite moment generating function $M_{T}(t)$ for $t > 2
\log \mu$. Let $\sigma^{2} = N_{E}p(1-p)$ and $\mu = N_{E}p$. We have
$$
  \mbox{CV}(S_{T}) \geq \frac{\sigma\sqrt{\mu - 1}}{\mu
       \sqrt{\mu + 1}}  \,.
$$
\end{pro}
\begin{proof}
The proof follows from straightforward elementary calculations. By the property of the Galton-Watson process, we
have
$$
  \mathbb{E}[Z_{n}] = \mu^{n}
$$
and
$$
  \mbox{Var}[Z_{n}] = \mu^{n-1}\sigma^{2}(1 + \mu + \cdots +
  \mu^{n-1}) \geq \mu^{2n-2}\sigma^{2} \,.
$$
In addition, for $n \geq m$ we have
\begin{eqnarray*}
  \mbox{cov}(Z_{m}, Z_{n}) &=& = \mathbb{E}[ Z_{m}Z_{n}] -
                               \mathbb{E}[Z_{m}] \mathbb{E}[Z_{n}]\\
&=& \mu^{n - m} \mathbb{E}[Z_{m}^{2}] - \mu^{m+n}\\
&\geq& \mu^{n-m} \mathbb{E}[Z_{m}]^{2} - \mu^{m+n} = 0 \,
\end{eqnarray*}
Therefore, we have
$$
  \mathbb{E}[S_{n}] = \sum_{k = 1}^{n} \mu^{k} = \frac{\mu^{n+1} -
    1}{\mu - 1} 
$$
and
\begin{eqnarray*}
\mbox{Var}[S_{n}] & \geq  & \sum_{k = 1}^{n} \mbox{Var}[Z_{k}]  \\
&=& \sigma^{2}(1 + \cdots + \mu^{2n - 2}) = \sigma^{2}\frac{\mu^{2n} -
    1}{\mu^{2} - 1} \,.
\end{eqnarray*}

By the law of total expectation, 
$$
  \mathbb{E}[S_{T}] = \mathbb{E}[\frac{\mu^{T+1} -
    1}{\mu - 1} ] =
  \frac{\mu M_{T}(\log \mu) - 1}{\mu - 1}\,.
$$
By the law of total variance we have
$$
  \mbox{Var}[S_{T}] \geq \mathbb{E}[\mbox{Var}[S_{T}|T]] \geq
  \frac{\sigma^{2}}{\mu^{2} - 1} \mathbb{E}[ \mu^{2T} - 1] =
  \frac{\sigma^{2}}{\mu^{2} - 1} (M_{T}(2 \log \mu) - 1) \,.
$$
By Jensen's inequality 
$$
  M_{T}(2 \log \mu) = \mathbb{E}[ (\mu^{T})^{2}] \geq
  (\mathbb{E}[\mu^{T}])^{2} = M_{T}(\log \mu)^{2} \,.
$$
Let $C = M_{T}(\log \mu) > 1$. We have
\begin{eqnarray*}
  \mbox{CV}(S_{T}) &\geq& \sigma \frac{\mu - 1}{\mu \sqrt{\mu^{2} - 1}}
  \cdot \frac{ \sqrt{C^{2} - 1}}{C - \mu^{-1}}\\
&\geq&\sigma \frac{\mu - 1}{\mu \sqrt{\mu^{2} - 1}} \cdot
       \sqrt{\frac{C+1}{C - 1}} \geq \frac{\sigma\sqrt{\mu - 1}}{\mu
       \sqrt{\mu + 1}} \,.
\end{eqnarray*}
\end{proof}

Note that the effective $p$ is usually a small number such that
$N_{E}p = O(1)$. For example, if membrane potential is uniformly
distributed then the effective $\mu = N_{E}p$ is $N_{E}P_{EE}S_{EE}/M =
2.25$. This corresponds to $\frac{\sigma\sqrt{\mu - 1}}{\mu
       \sqrt{\mu + 1}} \approx 0.4$. This
branching process approximation fails when the probability of $S_{T}>
N_{E}$ can not be neglected. If this happens, the coefficient of
variation will be significantly smaller as the branching process has
to stop when reaching $N_{E}$.

\section{Conclusion}

As introduced in Section 1, we study a stochastic model that models
a large and heterogeneous brain area that contains  many local
populations. Each local population has many densely connected
excitatory and inhibitory neurons. In addition, nearest-neighbor local
populations are connected. One can treat a local population as an orientation
hypercolumn of the primary visual cortex. Similar to our previous
paper \cite{li2017well}, one salient feature of this model is the multiple
firing event, in which a proportion of neurons (but not all) in the
population spike in a small time window. Multiple firing event is a
neuronal activity that lies between synchronization and homogeneous
spiking, which is widely believed to be related to the Gamma rhythm in
the cortex. 

After proving the stochastic stability, we proposed two mean-field
approximations that comes from simple linear ODE models. Both
approximations are exactly solvable. The common assumption in
these approximations is that excitatory and inhibitory spikes are
produced in a time-homogeneous way. Then we studied the discrepancies
of these mean-field approximations. Similar as in \cite{li2017well}, a
neuron will miss many incoming postsynaptic kicks when it stays at the
refractory state. When the population has synchronized spiking
activity, a significant amount of current is missed. The composition
of excitatory and inhibitory missing current is not proportional to
the total excitatory and inhibitory current, which affects the mean
firing rate and causes discrepancies of the mean-field
approximations. In particular, when the network is very synchronized,
the mean-field approximation fails. One obtains a better firing rate
prediction by assuming that the network is totally synchronized and
periodic. 

Then we demonstrated the decay of spatial correlation in the
model. Our simulation shows that  correlated multiple firing events
can usually spread to several local populations away. This is
consistent with the physiological fact that the Gamma rhythm is
usually very local. We then constructed an ODE model to describe what
happens during a multiple firing event. This ODE model explains why a
multiple firing event in a local population could induce spike volleys
in its neighbor local populations. Further, we found that unless the
multiple firing event is so strong that it becomes a synchronized
spiking event, the spike count of a multiple firing event has a very
high diversity. As a result, voltage profiles after a multiple
firing event are very different among local populations, even if the
external drive rate is very homogeneous. This mechanism can be modeled
by a branching process that is stopped at a random time. Our numerical
simulation shows that the diversity of multiple firing event at least
partially explains the mechanism of the decay of the spatial
correlation.

\appendix
\section{Proof of Theorem \ref{limitcase}.}
\begin{proof}
Let $\delta_{E} = c_{E}S_{EE}P_{EE}$, $\delta_{I} = c_{E}S_{IE}P_{IE}$, 
$\tau_{E} = 0$, $\tau_E = 1$, $\tau_{I} =\infty$, and $\lambda_{E} = \lambda_{I } =
0$. Let $(H_{0}, 0, 0, 0, 0, 0)$ be the initial condition. Then the
ODE system becomes
\begin{eqnarray*}
\frac{\mathrm{d}H_{E}}{\mathrm{d}t} & =& - H_{E} + P_{EE}H_{E}G_{E} \\
\frac{\mathrm{d}G_{E}}{\mathrm{d}t} &=& \delta_{E}  H_{E}(N_{E} - G_{E} - R_{E}) - P_{EE}H_{E}G_{E}\\
\frac{\mathrm{d}R_{E}}{\mathrm{d}t}&=&P_{EE}H_{E}G_{E}\\
\frac{\mathrm{d}H_{I}}{\mathrm{d}t}& =& - H_{I} + P_{IE}H_{E}G_{I} \\
\frac{\mathrm{d}G_{I}}{\mathrm{d}t} &=& \delta_{I}  H_{E}(N_{I} - G_{I} - R_{I}) - P_{IE}H_{E}G_{I}\\
\frac{\mathrm{d}R_{I}}{\mathrm{d}t}&=&P_{IE}H_{E}G_{I} \,.
\end{eqnarray*}
Let $u_{E} = G_{E} + R_{E}$ and $v_{E} = H_{E} - R_{E}$, we have
\begin{eqnarray*}
\frac{\mathrm{d}H_{E}}{\mathrm{d}t} &  = & -H_{E} + P_{EE}H_{E}(u_{E} + v_{E} - H_{E}) \\
\frac{\mathrm{d} u_{E}}{ \mathrm{d} t} &=& \delta_{E}  H_{E}(N_{E} - u_{E})\\
\frac{ \mathrm{d} v_{E}}{ \mathrm{d}t} &=& -H_{E} \,, 
\end{eqnarray*}
with $H_{E}(0) = v_{E}(0) = H_{0}$ and $u_{E}(0) = 0$. Divide $\mathrm{d}u_{E}/ \mathrm{d}t$ by $\mathrm{d}v_{E}/ \mathrm{d}t$,
we have 
$$
  \frac{\mathrm{d}u_{E}}{ \mathrm{d} v_{E}} = \delta_{E}  (u_{E} - N_{E}), \quad u_{E}(H_{0})
  = 0\,.
$$
Solving this initial value problem, one obtains
$$
  u_{E} (v_{E})= N_{E}( 1 - e^{\delta_{E} (v_{E} - H_{0})}) \,.
$$
Similarly, divide $\mathrm{d}H_{E}/ \mathrm{d}t$ by $\mathrm{d}v_{E}
/\mathrm{d}t$, we have
$$
  \frac{\mathrm{d}H_{E}}{\mathrm{d}v_{E}} = H_{E} + 1 - u_{E} - v_{E}
  , \quad H_{E}(H_{0}) =
  H_{0} \,.  
$$
This is a first order linear equation. The solution is 
$$
  H_{E}(v_{E}) = N_{E} + v_{E} + \frac{N_{E}}{\delta_{E} - P_{EE}}(P_{EE}e^{\delta_{E} (v_{E} - H_{0})} -
  \delta_{E} e^{P_{EE}(v_{E} - H_{0})}) \,.
$$
Therefore, equation 
$$
  \frac{\mathrm{d}v_{E}}{\mathrm{d}t} = - H_{E}(v_{E}), \quad v_E(0) = H_0
$$
becomes an autonomous equation. Let $- R^{*}$ be the greatest root of $H_{E}(v_{E})$ that is less than
$H_{0}$. It is easy to check that $R_{E}^{*} < N_{E}$. Hence as $t \rightarrow \infty$ we have $v_{E}(t) \rightarrow
R^{*}$. This implies $H_{E}(t) \rightarrow 0$ and $R_{E}(t) \rightarrow R^{*}$ as $t \rightarrow
\infty$. 

It remains to estimate $R^{*}$. We have
$$
  H_{E}(-N_{E}) \leq \frac{N_{E}}{\delta_{E} - P_{EE}}(P_{EE}
  e^{\delta_{E}(-N_{E} - H_{0})} - \delta_{E} e^{ (-N_{E} -
    H_{0})}) < 0\ ,. 
$$
On the other hand, let $A = N_{E} + H_{0} - \alpha$, we have
$$
  H_{E}( - N_{E} +\alpha) = \alpha + N_{E}e^{-P_{EE}A} +
  \frac{P_{EE}N_{E}}{\delta_{E} - P_{EE}}(e^{-\delta_{E}A} -
  e^{-P_{EE}A}) \,.  
$$
By mean value theorem, we have
$$
  H_{E}( - N_{E} +\alpha) \geq \alpha + N_{E}e^{-P_{EE}A} -
  P_{EE}N_{E}Ae^{-A m} \,,
$$
where $m = \min\{\delta_{E}, P_{EE}\}$. Since $A > N_{E}$ by the
assumption of the theorem, we have
$$
  H_{E}( - N_{E} +\alpha) \geq \alpha - P_{EE}N_{E}^{2}e^{-m N_{E}} >0
$$
provided $N_{E} m > 1$. By the intermediate value theorem, $R_{*}$
must be between $N_{E}$ and $N_{E} - \alpha$. This implies
$R_{E}(\infty) = R_{*} > N_{E} - \alpha$. 

\medskip

It remains to check $R_{I}(\infty)$. Let $B = R^{*} + H_{0}$. Recall that we have
$$
  u_{E}(\infty) = N_{E}(1 - e^{-\delta_{E}B}) \,.
$$
On the other hand, if we treat $H_{E}(t)$ as a time-dependent
variable, we have
$$
  u_{E}(\infty) = N_{E}(1 - e^{-\delta_{E} \int_{0}^{\infty} H_{E}(s)
    \mathrm{d}s}) \,.
$$
This implies
$$
  \int_{0}^{\infty}H_{E}(s) \mathrm{d}s = B \,.
$$
Now let $u_{I} = R_{I} + G_{I}$, we have
$$
  \frac{\mathrm{d}u_{I}}{\mathrm{d}t} = \lambda_{I}H_{E}(t)(N_{I} -
  u_{I}) \,,
$$
which is again a separable equation. This implies
$$
  u_{I}(\infty) = N_{I}(1 - e^{-\delta_{I}B}) \,.
$$
Finally, we have
$$
  \frac{\mathrm{d}G_{I}}{\mathrm{d}t} =
  \frac{\mathrm{d}u_{I}}{\mathrm{d}t} - P_{IE}H_{E}G_{I} \,,
$$
which is a first order linear equation. Consider the initial condition
$G_{I}(0) = u_{I}(0) = 0$, we have
$$
  e^{P_{IE}\int_{0}^{\infty} H_{E}(s) \mathrm{d}s}G_{I}(\infty) =
  u_{I}(\infty) \,.
$$
Therefore, we have
$$
  R_{I}(\infty) = N_{I} - N_{I}e^{-\delta_{I}B} - e^{-P_{IE}B}N_{I}(1
  - e^{-\delta_{I}B}) > N_{I} - N_{I}(e^{-\delta_{I}B} + e^{-P_{IE}B}) \,.
$$
Since $B = R^{*} + H_{0} > N_{E} - \alpha + H_{0} > N_{E}$, we have
$$
  R_{I}(\infty) \geq N_{I} - N_{I}(e^{-\delta_{I}N_{E}} +
  e^{-P_{IE}N_{E}}) = N_{I} - \beta \,.
$$
The theorem then follows by the definition of limit and the continuous
dependency of the solution on parameters. 
\end{proof}
\section{Multiple firing event equation for many local populations} 
Let $m = 1, \cdots, M$ and $n = 1, \cdots, N$ be indice of local
populations. The multiple firing event equation contains variables
$H^{m,n}_{E}, G^{m,n}_{E}, R^{m,n}_{E}, H^{m,n}_{I}, G^{m,n}_{I},
R^{m,n}_{I}$, whose roles are the same as in equation \eqref{MFE}. For
the sake of simplicity denote
$$
  J^{m,n}_{Q_{1}Q_{2}} = P_{Q_{1}Q_{2}}H^{m,n}_{Q_{2}} +
  \rho_{Q_{1}Q_{2}}\sum_{(m',n') \in
    \mathcal{N}(m,n)}H^{m',n'}_{Q_{2}} \,.
$$
For each $(m,n)$, the time evolution of variables $H^{m,n}_{E}, G^{m,n}_{E}, R^{m,n}_{E}, H^{m,n}_{I}, G^{m,n}_{I},
R^{m,n}_{I}$ are given by equations
\begin{equation}
\label{MFE2}
\begin{split}
\frac{\mathrm{d} H^{m,n}_{E}}{\mathrm{d} t} = & - \tau_{E}^{-1} H^{m,n}_{E} +
                                           \tau_{E}^{-1}J^{m,n}_{EE}G^{m,n}_{E}
  + \frac{\lambda_{E}}{S_{EE}} G^{m,n}_{E}\\
\frac{\mathrm{d} G^{m,n}_{E}}{\mathrm{d}t} =& c_{E} \max\{
                                         \tau_{E}^{-1}S_{EE}J^{m,n}_{EE}
                                         + \lambda_{E}
                                         - \tau_{I}^{-1}S_{EI}J^{m,n}_{EI},
                                         0\}\cdot ( N_{E} -
                                         G^{m,n}_{E} - R^{m,n}_{E}) -\\
&
                                         \tau_{I}^{-1}\max\{\frac{S_{EI}}{S_{EE}},
                                         1\}J^{m,n}_{EI}G^{m,n}_{E} -
   (\tau_{E}^{-1}J^{m,n}_{EE} + \frac{\lambda_{E}}{S_{EE}})G^{m,n}_{E}
                                        \\ 
\frac{\mathrm{d} R^{m,n}_{E}}{\mathrm{d}t}=&
                                        \tau_{E}^{-1}P_{EE}J^{m,n}_{EE}
                                        +  \frac{\lambda_{E}}{S_{EE}} G^{m,n}_{E}\\
\frac{\mathrm{d} H^{m,n}_{I}}{\mathrm{d} t} = & - \tau_{I}^{-1} H^{m,n}_{I} +
                                           \tau_{E}^{-1}J^{m,n}_{IE}G^{m,n}_{I}
  + \frac{ \lambda_{I}}{S_{IE}}G^{m,n}_{I}\\
\frac{\mathrm{d} G^{m,n}_{I}}{\mathrm{d}t} =& c_{I} \max\{
                                         \tau_{E}^{-1}S_{IE}J^{m,n}_{IE}
                                         + \lambda_{I}
                                         -
                                         \tau_{I}^{-1}S_{II}J^{m,n}_{II},
                                         0\}\cdot ( N_{I} -
                                         G^{m,n}_{I} - R^{m,n}_{I}) \\
& -\tau_{I}^{-1}\max\{ \frac{S_{II}}{S_{IE}}, 1\}J^{m,n}_{II}G^{m,n}_{I} -
   (\tau_{E}^{-1}J^{m,n}_{IE} + \frac{\lambda_{I}}{S_{IE}})G^{m,n}_{I}\\
\frac{\mathrm{d} R^{m,n}_{I}}{\mathrm{d}t}=&
                                        \tau_{E}^{-1}J^{m,n}_{IE}G^{m,n}_{I}
                                        + \frac{
                                          \lambda_{I}}{S_{IE}}G^{m,n}_{I} \,.
\end{split}
\end{equation}

\bibliography{myref}{}

\providecommand{\bysame}{\leavevmode\hbox to3em{\hrulefill}\thinspace}
\providecommand{\MR}{\relax\ifhmode\unskip\space\fi MR }
\providecommand{\MRhref}[2]{%
  \href{http://www.ams.org/mathscinet-getitem?mr=#1}{#2}
}
\providecommand{\href}[2]{#2}
\begin{thebibliography}{10}

\bibitem{borgers2005background}
Christoph B{\"o}rgers, Steven Epstein, and Nancy~J Kopell, \emph{Background
  gamma rhythmicity and attention in cortical local circuits: a computational
  study}, Proceedings of the National Academy of Sciences of the United States
  of America \textbf{102} (2005), no.~19, 7002--7007.

\bibitem{borgers2003synchronization}
Christoph B{\"o}rgers and Nancy Kopell, \emph{Synchronization in networks of
  excitatory and inhibitory neurons with sparse, random connectivity}, Neural
  computation \textbf{15} (2003), no.~3, 509--538.

\bibitem{borgers2005effects}
\bysame, \emph{Effects of noisy drive on rhythms in networks of excitatory and
  inhibitory neurons}, Neural computation \textbf{17} (2005), no.~3, 557--608.

\bibitem{cai2006kinetic}
David Cai, Louis Tao, Aaditya~V Rangan, David~W McLaughlin, et~al.,
  \emph{Kinetic theory for neuronal network dynamics}, Communications in
  Mathematical Sciences \textbf{4} (2006), no.~1, 97--127.

\bibitem{cai2004effective}
David Cai, Louis Tao, Michael Shelley, and David~W McLaughlin, \emph{An
  effective kinetic representation of fluctuation-driven neuronal networks with
  application to simple and complex cells in visual cortex}, Proceedings of the
  National Academy of Sciences of the United States of America \textbf{101}
  (2004), no.~20, 7757--7762.

\bibitem{chariker2015emergent}
Logan Chariker and Lai-Sang Young, \emph{Emergent spike patterns in neuronal
  populations}, Journal of computational neuroscience \textbf{38} (2015),
  no.~1, 203--220.

\bibitem{goddard2012gamma}
C~Alex Goddard, Devarajan Sridharan, John~R Huguenard, and Eric~I Knudsen,
  \emph{Gamma oscillations are generated locally in an attention-related
  midbrain network}, Neuron \textbf{73} (2012), no.~3, 567--580.

\bibitem{hairer2011yet}
Martin Hairer and Jonathan~C Mattingly, \emph{Yet another look at harris’
  ergodic theorem for markov chains}, Seminar on Stochastic Analysis, Random
  Fields and Applications VI, Springer, 2011, pp.~109--117.

\bibitem{haskell2001population}
Evan Haskell, Duane~Q Nykamp, and Daniel Tranchina, \emph{A population density
  method for large-scale modeling of neuronal networks with realistic synaptic
  kinetics}, Neurocomputing \textbf{38} (2001), 627--632.

\bibitem{henrie2005lfp}
J~Andrew Henrie and Robert Shapley, \emph{Lfp power spectra in v1 cortex: the
  graded effect of stimulus contrast}, Journal of neurophysiology \textbf{94}
  (2005), no.~1, 479--490.

\bibitem{hubel1995eye}
David~H Hubel, \emph{Eye, brain, and vision.}, Scientific American
  Library/Scientific American Books, 1995.

\bibitem{kaschube2010universality}
Matthias Kaschube, Michael Schnabel, Siegrid L{\"o}wel, David~M Coppola,
  Leonard~E White, and Fred Wolf, \emph{Universality in the evolution of
  orientation columns in the visual cortex}, science \textbf{330} (2010),
  no.~6007, 1113--1116.

\bibitem{lee2003synchronous}
Kwang-Hyuk Lee, Leanne~M Williams, Michael Breakspear, and Evian Gordon,
  \emph{Synchronous gamma activity: a review and contribution to an integrative
  neuroscience model of schizophrenia}, Brain Research Reviews \textbf{41}
  (2003), no.~1, 57--78.

\bibitem{li2017well}
Yao Li, Logan Chariker, and Lai-Sang Young, \emph{How well do reduced models
  capture the dynamics in models of interacting neurons?}, arXiv preprint
  arXiv:1711.01487 (2017).

\bibitem{menon1996spatio}
V~Menon, WJ~Freeman, BA~Cutillo, JE~Desmond, MF~Ward, SL~Bressler, KD~Laxer,
  N~Barbaro, and AS~Gevins, \emph{Spatio-temporal correlations in human gamma
  band electrocorticograms}, Electroencephalography and clinical
  Neurophysiology \textbf{98} (1996), no.~2, 89--102.

\bibitem{meyn2009markov}
Sean~P Meyn and Richard~L Tweedie, \emph{Markov chains and stochastic
  stability}, Cambridge University Press, 2009.

\bibitem{rangan2013dynamics}
Aaditya~V Rangan and Lai-Sang Young, \emph{Dynamics of spiking neurons: between
  homogeneity and synchrony}, Journal of Computational Neuroscience \textbf{34}
  (2013), no.~3, 433--460.

\bibitem{rangan2013emergent}
\bysame, \emph{Emergent dynamics in a model of visual cortex}, Journal of
  Computational Neuroscience \textbf{35} (2013), no.~2, 155--167.

\bibitem{stout1974almost}
William Stout, \emph{Almost sure convergence}, vol.~95, Academic Press, 1974.

\bibitem{wilson1972excitatory}
Hugh~R Wilson and Jack~D Cowan, \emph{Excitatory and inhibitory interactions in
  localized populations of model neurons}, Biophysical journal \textbf{12}
  (1972), no.~1, 1--24.

\bibitem{wilson1973mathematical}
\bysame, \emph{A mathematical theory of the functional dynamics of cortical and
  thalamic nervous tissue}, Biological Cybernetics \textbf{13} (1973), no.~2,
  55--80.

\end{thebibliography}
\bibliographystyle{amsplain}
\end{document}